\documentclass[11pt]{amsart}

\usepackage[english]{babel}
\usepackage[headings]{fullpage}
\usepackage{amsmath,amsfonts,amsthm,amssymb,amscd,mathrsfs,enumerate}
\usepackage[hidelinks]{hyperref}

\DeclareMathOperator{\supp}{supp}

\DeclareMathOperator*{\essinf}{ess\,inf}
\DeclareMathOperator{\RH}{RH}
\DeclareMathOperator{\diam}{diam}
\DeclareMathOperator{\divv}{div}
\DeclareMathOperator{\re}{Re}

\newtheorem{theorem}{Theorem}[section]
\newtheorem{lemma}[theorem]{Lemma}
\newtheorem{proposition}[theorem]{Proposition}

\theoremstyle{definition}
\newtheorem{definition}[theorem]{Definition}
\newtheorem{remark}[theorem]{Remark}
\newtheorem{example}[theorem]{Example}
\numberwithin{equation}{section}

\newcommand{\R}{\mathbf{R}}
\newcommand{\C}{\mathbf{C}}
\newcommand{\Z}{\mathbf{Z}}
\newcommand{\D}{\mathscr{D}}
\newcommand{\M}{\mathscr{M}}
\newcommand{\Sp}{\mathscr{S}}
\newcommand{\eps}{\varepsilon}
\newcommand{\calR}{\ensuremath{\mathcal{R}}}

\begin{document}
\pagestyle{headings}
\title[$A_1$--$A_\infty$ estimates for sparsely dominated operators]{Weak and strong type $A_1$--$A_\infty$ estimates for sparsely dominated operators}
\author{Dorothee Frey \and Zoe Nieraeth}

\address{Dorothee Frey, Delft Institute of Applied Mathematics, Delft University of Technology, P.O. Box 5031, 2600 GA Delft, The Netherlands\\}
\email{d.frey@tudelft.nl}

\address{Zoe Nieraeth, Delft Institute of Applied Mathematics, Delft University of Technology, P.O. Box 5031, 2600 GA Delft, The Netherlands\\}
\email{znieraeth@bcamath.org}

\date{\today}

\begin{abstract}
We consider operators $T$ satisfying a sparse domination property
\[
|\langle Tf,g\rangle|\leq c\sum_{Q\in\Sp}\langle f\rangle_{p_0,Q}\langle g\rangle_{q_0',Q}|Q|
\]
with averaging exponents $1\leq p_0<q_0\leq\infty$.

We prove weighted strong type boundedness for $p_0<p<q_0$ and use new techniques to prove weighted weak type $(p_0,p_0)$ boundedness with quantitative mixed $A_1$--$A_\infty$ estimates, generalizing results of Lerner, Ombrosi, and P\'erez and Hyt\"onen and P\'erez. Even in the case $p_0=1$ we improve upon their results as we do not make use of a H\"ormander condition of the operator $T$. Moreover, we also establish a dual weak type $(q_0',q_0')$ estimate.

In a last part, we give a result on the optimality of the weighted strong type bounds including those previously obtained by Bernicot, Frey, and Petermichl.
\end{abstract}

\keywords{Sparse domination, Muckenhoupt weights, sharp weighted bounds}

\maketitle

\section{Introduction}
In recent years, after a solution was found to the the well-known $A_2$ conjecture \cite{hytonen123}, the role of sparse operators has become increasingly important in the weighted theory of many operators, see for instance \cite{lacey17, lernaz15, condealonso16, diplinio16, benea171} and references therein. Sparse domination yields optimal quantitative $A_p$ estimates for $1<p<\infty$ for example for the classical Riesz transforms in $\R^n$. As has been shown by Bernicot, Frey, and Petermichl \cite{frey16}, the idea of sparse domination reaches far beyond the theory of Calder\'on-Zygmund operators. Indeed, one can consider the Riesz transform $\nabla L^{-1/2}$ in e.g. a convex doubling domain in $\R^n$, where $L$ is the Laplace operator with respect to Neumann boundary conditions. Generally, the Riesz transform in such a setting does not satisfy any pointwise regularity estimates and therefore falls outside of the class of Calder\'on-Zygmund operators.  However, it  satisfies a sparse domination property which does in fact yield the quantitative weighted bounds from the $A_2$ conjecture. In $\R^n$, foregoing the full range of $1<p<\infty$, one  can  consider the Riesz transform for elliptic operators $L=-\divv(A\nabla)$ for $A$ with bounded,  complex  coefficients. Such operators are only bounded in $L^p$ for a certain range $p_0<p<q_0$, and it was established in \cite{frey16} that they satisfy a sparse domination property
\[
|\langle Tf,g\rangle|\leq c\sum_{Q\in\Sp}\langle f\rangle_{p_0,Q}\langle g\rangle_{q_0',Q}|Q|
\]
from which general quantitative weighted bounds in the respective weighted $L^p$-spaces are deduced.

For Calder\'on-Zygmund operators, weighted weak type $(1,1)$ estimates were established by Lerner, Ombrosi, and P\'erez \cite{lerner09} and later improved upon by Hyt\"onen and P\'erez \cite{hytonen13}.  In this article,  we establish the corresponding $(p_0,p_0)$ estimate in the  above described  more general setting. The arguments used in \cite{lerner09} rely on introducing weights in the classical arguments involving Calder\'on-Zygmund decompositions $f=g+b$ and the vanishing mean value property of the `bad' part $b$ in combination with the H\"ormander condition of the kernel of the operator. In general, the operators we are considering here need not be integral operators at all and for the more general operators such as the Riesz transform associated to an elliptic operator, an argument by Blunck and Kunstmann \cite{kunstmann03}  (see also \cite{HM03})  gave a weak type $(p_0,p_0)$ boundedness using an adapted $L^{p_0}$ Calder\'on-Zygmund decomposition, where a certain cancellation of the operator with respect to the semigroup generated by the elliptic operator replaces the regularity estimates of the kernel. Weights were then introduced into this argument by Auscher and Martell \cite{auscher072}, but it seems like these techniques do not yield optimal bounds in terms of the constants of the weights. Therefore, we give a new argument to establish the corresponding bounds while still recovering the old bounds found in \cite{hytonen13}.

Here, in order to combine the previous approaches and to tie the theory together, we deduce quantitative weighted bounds directly from sparse domination assumptions. We introduce weights into a weak boundedness argument for sparse operators where there exists a Calder\'on-Zygmund decomposition  with the property that  the `bad' part $b$ cancels completely. We then combine this with generalizations of the main lemmata used in \cite{lerner09}. Moreover, we leave the Euclidean setting and extend the results to more general doubling metric measure spaces including certain bounded domains and Riemannian manifolds as was also studied in \cite{kunstmann03} and \cite{auscher072, auscher08}.

 In a last part we show that the strong type weighted estimates are optimal, given a precise control of the asymptotic behaviour of the unweighted $L^p$ operator norm of $T$ at the endpoints $p=p_0$ and $p=q_0$. We give an example of such an operator in the case $p_0=1$, $q_0=n$.

\subsection{The setting}
We consider the Euclidean space $\R^n$ equipped with a Borel measure $\mu$ that satisfies $0<\mu(B)<\infty$ for all balls $B$ and which satisfies the doubling property, i.e., there is a $C>0$ such that
\begin{equation}\label{eq:doublingprop}
\mu(2B)\leq C\mu(B)
\end{equation}
for all balls $B$, where $2B$ denotes the ball with the same center as $B$ and whose radius is twice that of the radius of $B$. Taking the smallest such $C$ we define $\nu:=\log_2 C$, which we refer to as the \textit{doubling dimension}. We write $|E|:=\mu(E)$ and for each measurable set $E$ of finite non-zero measure and each $0<p\leq\infty$ we will write
\[
\langle f\rangle_{p,E}:=\|f\chi_E\|_{L^p(|E|^{-1}\mathrm{d}\mu)},
\]
where $\chi_E$ denotes the indicator function of the set $E$. We write $\langle f,g\rangle:=\int f g\,\mathrm{d}\mu$, and define $p'=p/(p-1)\in[1,\infty]$ for $1\leq p\leq\infty$.

For $\alpha\in\big\{0,\frac{1}{3},\frac{2}{3}\big\}^n$ we will consider the translated dyadic systems
\[
\D^\alpha:=\bigcup_{k\in\Z}\big\{2^{-k}\big([0,1)^n+m+(-1)^k\alpha\big):m\in\Z^n\big\},
\]
and $\D:=\cup_\alpha\D^\alpha$.

The collection $\D$ is used as a replacement for the collection of balls or the collection of all cubes in $\R^n$, which is justified by the fact that for any ball $B(x;r)\subseteq\R^n$ there is a cube $Q\in\D$ so that $B(x;r)\subseteq Q$ and $\diam(Q)\leq \rho r$ for a constant $\rho=\rho(n)>0$, and for any cube $P\subseteq\R^n$ there is a cube $Q\in\D$ such that $P\subseteq Q$ and $\ell(Q)\leq 6\ell(P)$, where $\ell(R)$ denotes the side length of a cube $R$.

We say that a collection $\Sp\subseteq\D$ is called \emph{$\eta$-sparse} for $0<\eta\leq1$ if for each $\alpha\in\big\{0,\frac{1}{3},\frac{2}{3}\big\}^n$ there is a pairwise disjoint collection $(E_Q)_{Q\in\Sp\cap\D^\alpha}$ of measurable sets so that $E_Q\subseteq Q$ and $|Q|\leq\eta^{-1}|E_Q|$.

\begin{remark}\label{rem:homrem}
Since $\R^n$ is connected and unbounded, the doubling property implies that $\mu(\R^n)=\infty$ \cite{grig91}. We are working in $\R^n$ for notational reasons only; since our applications lie in a more general framework, our arguments are written so that they work with minimal adaptations in general doubling metric measure spaces $X$. Our main results remain true even when $\mu(X)<\infty$, for example when $X$ is a bounded Lipschitz domain in $\R^n$.  We will detail how this can be seen in Section \ref{sec:ext}.
\end{remark}
We let $\mathcal{D}$ be a space of test functions on $\R^n$ with the property that it is dense in $L^p(w)$ for all $1\leq p<\infty$ and all weights $w\in A_\infty$, such as, for example, $\mathcal{D}=C_c^\infty(\R^n)$.
\begin{definition}
Let $T$ be a (sub)linear operators, initially defined on $\mathcal{D}$, with the following property: There are $1\leq p_0<q<q_0\leq\infty$ and constants $c>0$ and $0<\eta\leq1$ so that for each pair of functions $f,g\in\mathcal{D}$ there is an $\eta$-sparse collection $\Sp\subseteq\D$ so that
\begin{equation}\label{eq:sumsparse}
|\langle Tf,g\rangle|\leq c\sum_{Q\in\Sp}\langle f\rangle_{p_0,Q}\langle g\rangle_{q_0',Q}|Q|.
\end{equation}
Then we will write $T\in S(p_0,q_0)$, or $T\in S(p_0,q_0;\mu)$ if we wish to emphasize the underlying measure, and we shall refer to the operators in this class as \textit{sparsely dominated operators}.
\end{definition}
If $T\in S(p_0,q_0)$, then it extends to a bounded operator on $L^p$ for all $p_0<p<q_0$, see Proposition \ref{prop:sparseim} below. For examples of operators in this class we refer the reader to Subsection \ref{subsec:examples}

When writing that a constant $C=C(T)>0$ depends on $T$, we mean that it depends on the constants $c,\eta$ in the domination property \eqref{eq:sumsparse}. We remark that the sum on the right-hand side of \eqref{eq:sumsparse} can be split into $3^n$ sums by considering the different dyadic grids, simplifying the proofs by only having to consider a single dyadic grid at a time. Finally, we remark that if $T$ is linear, then  $T\in S(p_0,q_0)$ if and only if $T^\ast\in S(q_0',p_0')$, where $T^\ast$ denotes the dual operator of $T$.

We will write $A\lesssim B$ when there is a constant $C>0$, independent of the important parameters, so that $A\leq C B$. Moreover we write $A\simeq B$ if $A\lesssim B$ and $B\lesssim A$.
\subsection{Main results}
For $1\leq p_0<q_0\leq\infty$ we consider an operator $T\in S(p_0,q_0)$. Then $T$ will be of strong type $(p,p)$ for any $p_0<p<q_0$ and of weak type $(p_0,p_0)$, see Proposition \ref{prop:sparseim} below. As a matter of fact, $T$ will satisfy weighted boundedness for various classes of weights. It has been shown in \cite{frey16} that for $p_0<p<q_0$ and any $w\in A_{p/p_0}\cap\RH_{(q_0/p)'}$ we have
\begin{equation}\label{eq:atwo}
\|T\|_{L^p(w)\to L^p(w)}\lesssim[w^{(q_0/p)'}]_{A_{\phi(p)}}^{\max\left(\frac{1}{p-p_0},\frac{q_0-1}{q_0-p}\right)\big/\left(\frac{q_0}{p}\right)'}\leq\left([w]_{A_{p/p_0}}[w]_{\RH_{(q_0/p)'}}\right)^{\max\left(\frac{1}{p-p_0},\frac{q_0-1}{q_0-p}\right)},
\end{equation}
where $\phi(p)=(q_0/p)'(p/p_0-1)+1$, and that the exponent in the last estimate is optimal for sparse operators. This generalizes the positive result of the well-known $A_2$-conjecture, stating that for all Calder\'{o}n-Zygmund operators $T$ one has
\begin{equation}\label{eq:atwoclassical}
\|T\|_{L^p(w)\to L^p(w)}\lesssim[w]_{A_p}^{\max\left(\frac{1}{p-1},1\right)}.
\end{equation}
Indeed, the result in \eqref{eq:atwo} recovers this result since Calder\'{o}n-Zygmund operators are in the class $S(1,\infty)$. Historically, the estimate \eqref{eq:atwoclassical} was first proven to be true for the Beurling-Ahlfors transform by Petermichl and Volberg \cite{petermichl02}, solving an optimal regularity problem for solutions to Beltrami equations. In between this period and the time that \eqref{eq:atwoclassical} was established in full generality by Hyt\"{o}nen \cite{hytonen123}, it was shown by Lerner, Ombrosi, and P\'{e}rez \cite{lerner08} that for all Calder\'{o}n-Zygmund operators $T$ one has
\begin{equation}\label{eq:aone}
\|T\|_{L^p(w)\to L^p(w)}\lesssim p p'\log\left(e+\frac{1}{p-1}\right)[w]_{A_1}
\end{equation}
for all $1<p<\infty$, showing a significantly better exponent of the constant of the weight when considering the smaller class of weights $A_1\subseteq A_p$. Using mixed $A_1$--$A_\infty$ type estimes, this result was improved by Hyt\"{o}nen and P\'{e}rez \cite{hytonen13} to
\begin{equation}\label{eq:inftyonemix}
\|T\|_{L^p(w)\to L^p(w)}\lesssim pp'[w]_{A_\infty}^{\frac{1}{p'}}[w]^{^{\frac{1}{p}}}_{A_1},
\end{equation}
for $1<p<\infty$, where they are considering Wilson's $A_\infty$ constant
\[
[w]_{A_\infty}=\sup_{\text{$B$ a ball}}\frac{1}{w(B)}\int_B\!\M(w\chi_B)\,\mathrm{d}\mu,
\]
which appears in \cite{wilson87, wilson89, wilson08}. They also provided an improvement to \eqref{eq:atwoclassical} by using mixed $A_p$--$A_\infty$ type estimates. Such mixed type estimates have also appeared in the recent work by Li \cite{li17}, who gives a direct improvement of \eqref{eq:atwo}.

To continue on along this line of results, we establish the following:
\begin{theorem}\label{thm:main1}
Let $1\leq p_0<p<q_0\leq\infty$, $T\in S(p_0,q_0)$, and $w\in A_1\cap\RH_{(q_0/p)'}$. Then there is a constant $c=c(T,\nu,n)>0$ so that
\begin{equation}\label{eq:timate}
\|T\|_{L^p(w)\to L^p(w)}\leq c c_p[w^{(q_0/p)'}]_{A_\infty}^{\frac{1}{p'}}[w^{(q_0/p)'}]^{^{\frac{1}{p(q_0/p)'}}}_{A_1},
\end{equation}
with
\[
c_p=\left[\left(\frac{p'}{q_0'}\right)'\right]^{\frac{1}{q_0'}}\left[\left(\frac{p_0'}{p'}\right)'\left(\frac{p}{p_0}\right)'\right]^{\frac{1}{p_0}}.
\]
In particular, we have
\begin{equation}\label{eq:bestimate}
\|T\|_{L^p(w)\to L^p(w)}\lesssim c_p[w^{(q_0/p)'}]_{A_1}^{\frac{1}{q_0'}}\leq c_p\left([w]_{A_1}[w]_{\RH_{(q_0/p)'}}\right)^{\frac{q_0-1}{q_0-p}}.
\end{equation}
\end{theorem}
Our result \eqref{eq:timate} recovers \eqref{eq:inftyonemix} when setting $p_0=1$, $q_0=\infty$. One shows that \eqref{eq:bestimate} follows from \eqref{eq:timate} by applying \eqref{eq:ainftyregest} and Proposition \ref{prop:weightprop}(ii) below. This result recovers the exponent in \eqref{eq:aone} when $q_0=\infty$.

The constants found in the estimate \eqref{eq:timate} can be used to establish weighted weak type $(p_0,p_0)$ boundedness. In the work of Lerner, Ombrosi, and P\'{e}rez \cite{lerner08} it was shown that for all Calder\'{o}n-Zygmund operators $T$ and all weights $w\in A_1$ one has
\begin{equation}\label{eq:waone}
\|T\|_{L^1(w)\to L^{1,\infty}(w)}\lesssim [w]_{A_1}\log(e+[w]_{A_1}).
\end{equation}
This result is related to the weak Muckenhoupt-Wheeden conjecture, which is now known to be false \cite{nazarov10}, stating that one has linear dependence on $[w]_{A_1}$ on the right-hand side of \eqref{eq:waone}, and the logarithm can be removed. However, the result \eqref{eq:waone} was improved by Hyt\"{o}nen and P\'{e}rez \cite{hytonen13} to
\begin{equation}\label{eq:winftyonemix}
\|T\|_{L^1(w)\to L^{1,\infty}(w)}\lesssim [w]_{A_1}\log(e+[w]_{A_\infty}).
\end{equation}
It is expected that this dependence on the constants of the weight is optimal.

Both the proofs of \eqref{eq:waone} and \eqref{eq:winftyonemix} rely on taking a Calder\'on-Zygmund decomposition $f=g+b$. Here, the H\"{o}rmander condition of the kernel of $T$ is used to deal with the `bad' part $b$, using an argument that can already be found in \cite{perez94} (namely, they use \cite[Lemma 3.3,\,p.~413]{garcia85}). Since we are making no such assumptions on our operators, which may not even be integral operators, we rely on new methods to deal with this term, using only sparse domination. We establish the following result:
\begin{theorem}\label{thm:main2}
Let $1\leq p_0<p<q_0\leq\infty$, $T\in S(p_0,q_0)$, and $w\in A_1\cap\RH_{(q_0/p_0)'}$. Then there is a constant $c=c(T,p_0,q_0,\nu,n)>0$ so that
\[
\|T\|_{L^{p_0}(w)\to L^{p_0,\infty}(w)}\leq c\psi(w)
\]
with
\[
\psi(w)=\begin{cases}
[w]_{A_1}\log(e+[w]_{A_\infty}) & \text{if $p_0=1$, $q_0=\infty$;}\\
[w]^{\frac{1}{p_0}}_{A_1}[w]^{\frac{1}{p_0'}}_{A_\infty}\log(e+[w]_{A_\infty})^{\frac{2}{p_0}} & \text{if $p_0>1$, $q_0=\infty$;}\\
[w^{q_0'}]_{A_\infty}[w]_{A_1}[w]_{\RH_{q_0'}} & \text{if $p_0=1$, $q_0<\infty$;}\\
[w^{(q_0/p_0)'}]^{1+\frac{1}{p_0}}_{A_\infty}\left([w]_{A_1}[w]_{\RH_{(q_0/p_0)'}}\right)^{\frac{1}{p_0}}  & \text{if $p_0>1$, $q_0<\infty$.}
\end{cases}
\]
\end{theorem}
We note that in particular we recover the bound \eqref{eq:winftyonemix}. It is of interested to point out that we get this bound even for operators outside of the class of Calder\'on-Zygmund operators that are in $S(1,\infty)$, see Example \ref{ex:neumannlaplacian} below.

We also establish a dual result of the type first studied in \cite{lerner09}, generalizing the result \cite[Theorem 1.23]{hytonen13}. Here we denote by $T^\ast$ the dual operator of $T$ for linear $T$.
\begin{theorem}\label{thm:maindual}
Let $1\leq p_0<q_0\leq\infty$, $T\in S(p_0,q_0)$ linear and $w\in A_1$. Then there is a constant $c=c(T,p_0,q_0,\nu,n)>0$ so that
\[
\left\|\frac{T^\ast f}{w^{\frac{1}{q_0'}}}\right\|_{L^{q_0',\infty}(w)}\leq c([w]_{A_\infty}\log(e+[w]_{A_1}))^{\frac{1}{q_0'}}\|f\|_{q_0'}
\]
for all $f\in L^{q_0'}$.
\end{theorem}

 Using the ideas of \cite{luque15}, we then establish optimality of the weighted estimates in terms of the asymptotic behaviour of the unweighted $L^p$ operator norm of $T$ at the endpoints $p=p_0$ and $p=q_0$. We refer to  Definition \ref{def:exponents} for the definition of the exponents $\alpha_T(p_0)$ and $\gamma_T(q_0)$.

\begin{theorem} \label{thm:opti-end}
Let $1\leq p_0<q_0\leq\infty$, let $T\in S(p_0,q_0)$, and let $w\in A_{p/p_0}\cap\RH_{(q_0/p)'}$. Then the exponent in the estimate
\[
\|T\|_{L^p(w) \to L^p(w)}\lesssim[w^{(q_0/p)'}]_{A_{\phi(p)}}^{\max\left(\frac{1}{p-p_0},\frac{q_0-1}{q_0-p}\right)\big/\left(\frac{q_0}{p}\right)'}.
\]
from \cite{frey16} is optimal under the assumption that $\alpha_T(p_0)=1/p_0$ and $\gamma_T(q_0)=1/q_0'$.

Moreover, for $w\in A_1\cap\RH_{(q_0/p)'}$, the exponent in the estimate
\[
\|T\|_{L^p(w)\to L^p(w)}\lesssim[w^{(q_0/p)'}]_{A_1}^{\frac{1}{q_0'}}
\]
from Theorem \ref{thm:main1} is optimal under the assumption that $\gamma_T(q_0)=1/q_0'$.
\end{theorem}

In, e.g., the example of the Riesz transform on two copies of $\R^n$ glued smoothly along their unit circles \cite{CCH06}, it is known that $q_0=n$ and $\gamma_T(q_0)=(n-1)/n$, and thus the weighted estimate is optimal. See Example \ref{ex:twocopies}.

\subsection{Examples}\label{subsec:examples}
There is a wealth of examples of sparsely dominated operators. Other than the class of Calder\'on-Zygmund operators, our main examples can be found in \cite[Section 3]{frey16}. See also the earlier work \cite{auscher072}. We point out several examples of particular interest here.
\begin{example}[Riesz transform associated  with  elliptic second order divergence form operators]
Let $A$ be a complex, bounded,  measurable  matrix-valued function in $\R^n$ satisfying the ellipticity condition $\re(A(x)\xi\cdot\overline{\xi})\geq\lambda|\xi|^2$ for all $\xi\in\C^n$ and a.e. $x\in\R^n$. Then one can define a maximal accretive operator
\[
Lf:=-\divv(A\nabla f)
\]
which generates a semigroup $(e^{-tL})_{t>0}$. If both the semigroup and the family $(\sqrt{t}\nabla e^{-tL})_{t>0}$ satisfy $L^{p_0}$--$L^{q_0}$ off-diagonal estimates, then the Riesz transform $\mathscr{R}:=\nabla L^{-1/2}$ is in the class $S(p_0,q_0)$. In particular we point out that if we are using the Lebesgue measure in dimension $\nu=n=1$, we have $p_0=1$ and $q_0=\infty$ so that $\mathscr{R}\in S(1,\infty)$. We refer the reader to \cite{auscher07} for more values of $p_0$ and $q_0$ in other dimensions in the Euclidean setting and to \cite{frey16} for details on the sparse domination result.
\end{example}
\begin{example}[Riesz transform associated to Neumann Laplacian]\label{ex:neumannlaplacian}
Suppose $\Delta$ is the Laplace operator associated with Neumann boundary conditions in a bounded convex doubling domain in $\R^n$. As studied in \cite{wang13}, the Riesz transform $\nabla\Delta^{-1/2}$ will not in general have a kernel satisfying pointwise regularity estimates and is thus not in the class of Calder\'on-Zygmund operators. However, this operator does belong to the class $S(1,\infty)$ and will therefore satisfy the bound \eqref{eq:inftyonemix}. Note that for this example we need to apply our results to a metric measure space other than $\R^n$. We refer the reader to Section \ref{sec:ext} for an overview of the theory in bounded domains.
\end{example}
\begin{example}[Fourier multipliers]
Let $m$ be the function in $\R^n$ defined by $m(\xi)=1-|\xi|^2$ for $|\xi|\leq 1$ and $m(\xi)=0$ elsewhere. For $\delta\geq 0$, the Bochner-Riesz operator $\mathcal{B}^\delta$ is defined as the Fourier multiplier $\mathcal{B}^\delta f:=(m^\delta\hat{f})^{\vee}$. Then, for any $\delta>0$ there exists a $1<p_0<2$ so that for any $0<\eps<2-p_0$ we have $\mathcal{B}^\delta\in S(p_0+\eps, 2)$. For details we refer the reader to \cite{benea171}.
\end{example}
\section{Preliminaries}
\subsection{Notation}
For $1\leq p<\infty$ we denote by $\M_p$ the uncentered dyadic maximal operator
\[
\M_pf:=\sup_{Q\in\D^\alpha}\langle f\rangle_{p,Q}\chi_Q,
\]
where it will be made clear from the context which dyadic grid $\D^\alpha$ we are considering, and where we will write $\M:=\M_1$. Similarly we define $\M_p^{\mathscr{B}}$ and $\M^{\mathscr{B}}$ to be the uncentered maximal operators with respect to balls rather than cubes.

We list some of the basic definitions and facts about weights. A measurable function $w:\R^n\to(0,\infty)$ is called a weight. We identify a weight $w$ with a Borel measure by setting
\[
w(E):=\int_E\!w\,\mathrm{d}\mu
\]
for all measurable sets $E\subseteq\R^n$. For $1\leq p\leq\infty$ we respectively denote by $L^p(w)$ and $L^{p,\infty}(w)$ the Lebesgue and weak Lebesgue spaces with measure $w$.

For $1\leq p<\infty$ we say that $w\in A_p$ if
\[
[w]_{A_p}:=\sup_{Q\in\D}\langle w\rangle_{1,Q}\langle w^{-1}\rangle_{p'-1,Q}<\infty,
\]
where for $p=1$ we use the limiting interpretation $\langle w^{-1}\rangle_{p'-1,Q}=\langle w^{-1}\rangle_{\infty,Q}=(\essinf_Q w)^{-1}$. We say that $w\in A_\infty$ if its Wilson $A_\infty$ constant is finite, that is, if
\[
[w]_{A_\infty}:=\sup_{\text{$B$ a ball}}\frac{1}{w(B)}\int_B\!\M^{\mathscr{B}}(w\chi_B)\,\mathrm{d}\mu<\infty,
\]
where the supremum is taken over all balls $B\subseteq\R^n$. For an overview of this constant we refer the reader to \cite{hytonen122} and references therein. In particular we point out that for a dimensional constant $c=c(n,\nu)>0$ we have
\begin{equation}\label{eq:ainftyregest}
c[w]_{A_\infty}\leq[w]_{A_p}\leq[w]_{A_q},\quad 1\leq q<p<\infty,
\end{equation}
where the first inequality here can be found in \cite[Proposition 2.2]{hytonen13} while the second one follows from H\"{o}lder's inequality.

For $1<s\leq\infty$ we say that $w\in\RH_s$ if
\[
[w]_{\RH_s}:=\sup_{Q\in\D}\langle w\rangle_{s,Q}\langle w\rangle^{-1}_{1,Q}<\infty.
\]
For $s=1$ we will use the interpretation $\RH_1=A_\infty$, where we set $[w]_{\RH_1}:=1$.

We provide some facts about the classes $A_1$ and $A_\infty$ that we will use.
\begin{proposition}\label{prop:weightprop}
\begin{enumerate}[(i)]
\item $A_q=\bigcup_{1\leq p<q}A_p$ for $1<q\leq\infty$ and $\RH_s=\bigcup_{s<r\leq\infty}\RH_r$ for $1\leq s<\infty$. In particular we have $w\in A_\infty$ if and only if $w\in A_p$ for some $1\leq p<\infty$.
\item For $1\leq p<\infty$, $1\leq s<\infty$ we have $w\in A_p\cap\RH_s$ if and only if $w^s\in A_{s(p-1)+1}$. Moreover, we have
\[
[w^s]_{A_{s(p-1)+1}}\leq\left([w]_{A_p}[w]_{\RH_s}\right)^s.
\]
\item There are constants $c,\kappa>0$ depending only on the doubling dimension $\nu$, so that for every $w\in A_1$ we have
\[
\M^{\mathscr{B}}_q w\leq c[w]_{A_1}w\quad\text{for }1\leq q\leq 1+\frac{1}{\kappa[w]_{A_\infty}}.
\]
\end{enumerate}
\end{proposition}
\begin{proof}
For (i) we refer the reader to \cite{grafakosmodern, wilson87}. Property (ii) can be found in \cite{johnson91}.

Property (iii) is a consequence of \cite[Theorem 1.1]{hytonen122}. Indeed, this result states that there are constants $c,\kappa>0$ depending only on $\nu$ such that for any ball $B$ we have $\langle w\rangle_{q(w),B}\leq c\langle w\rangle_{1,2B}$, where $q(w):=1+1/(\kappa[w]_{A_\infty})$. Thus, (iii) follows from H\"older's inequality and the definition of $A_1$.
\end{proof}

\subsection{Weighted boundedness of sparsely dominated operators}
We wish to give some heuristic arguments as to why we can expect certain weighted boundedness of sparsely dominated operators. We start with the following observation:
\begin{proposition}\label{prop:sparseim}
Let $1\leq p_0<q_0\leq\infty$ and $T\in S(p_0,q_0)$. Then $T$ is of strong type $(p,p)$ for all $p_0<p<q_0$ and of weak type $(p_0,p_0)$.
\end{proposition}
The verification of the strong boundedness is by now standard, see also \cite{cruz12}. While the weak type boundedness should be well-known, we could not find a precise reference for the cases where $p_0>1$. For the case $p_0=1$ we refer the reader to \cite[Theorem E]{condealonso16}, see also \cite[Proposition 6]{benea17}. For completeness we give a proof of the general case here, which we defer to the end of this section.

We will show that if an operator $T$ lies in $S(p_0,q_0;\mu)$, then $T$ must also lie in $S(q_-,q_+;w)$ for appropriate weights $w$, and for certain $q_-<q_+$ depending on $w$. Then Proposition \ref{prop:sparseim} implies that $T$ satisfies weighted boundedness.

First we note that if we have a sparse collection $\Sp\subseteq\D$ with respect to the reference measure $\mu$, then $\Sp$ is also sparse with respect to all weights $w\in A_\infty$. Indeed, suppose $w\in A_p$ for some $1\leq p<\infty$ and suppose $\Sp$ is $\eta$-sparse with $(E_Q)_{Q\in\Sp\cap\D^\alpha}$ as one of the associated pairwise disjoint collections. Then, by H\"{o}lder's inequality (use the first equation in \eqref{eq:holweight} below with $f=\chi_{E_Q}$, $p=1$),
\[
w(Q)\leq\left(\frac{|Q|}{|E_Q|}\right)^p[w]_{A_p}w(E_Q)\leq\eta^{-p}[w]_{A_p}w(E_Q).
\]
Hence, $\Sp$ is $[w]_{A_p}^{-1}\eta^p$-sparse with respect to the measure $w$ with the same collections $(E_Q)_{Q\in\Sp\cap\D^\alpha}$.

Next we observe that for any $1\leq p\leq q<\infty$ it follows from H\"{o}lder's inequality that
\begin{equation}\label{eq:holweight}
\begin{split}
\langle f\rangle_{p,Q}&\leq[w]_{A_{q/p}}^{\frac{1}{q}}\left(\frac{1}{w(Q)}\int_Q\!|f|^q w\,\mathrm{d}\mu\right)^{\frac{1}{q}},\\
\langle g w\rangle_{q',Q}|Q|&\leq[w]_{\RH_{(q/p)'}}^{\frac{1}{p}}\left(\frac{1}{w(Q)}\int_Q\!|g|^{p'} w\,\mathrm{d}\mu\right)^{\frac{1}{p'}}w(Q).
\end{split}
\end{equation}
Thus, if $T\in S(p_0,q_0;\mu)$ and $w\in A_{p_1/p_0}\cap\RH_{(q_0/q_1)'}$ for some $p_0< p_1\leq q_1<q_0$, then it follows from the self-improvement properties (i) of Proposition \ref{prop:weightprop} that we can find $p_0\leq q_-< p_1$, $q_1<q_+\leq q_0$ so that $w\in A_{q_-/p_0}\cap\RH_{(q_0/q_+)'}$. Picking appropriate functions $f$, $g$, and by applying the sparse domination property to the pair $f$, $gw$, we find a sparse collection $\Sp\subseteq\D$ so that by \eqref{eq:holweight} we have
\[
\left|\int(Tf)gw\,\mathrm{d}\mu\right|\lesssim\sum_{Q\in\Sp}\left(\frac{1}{w(Q)}\int_Q\!|f|^{q_-} w\,\mathrm{d}\mu\right)^{\frac{1}{q_-}}\left(\frac{1}{w(Q)}\int_Q\!|g|^{q_+'} w\,\mathrm{d}\mu\right)^{\frac{1}{q_+'}}w(Q).
\]
In other words, we have $T\in S(q_-,q_+;w)$ and thus we obtain the boundedness
\[
T:L^p(w)\to L^p(w),\quad w\in A_{p_1/p_0}\cap\RH_{(q_0/q_1)'},\quad p_1\leq p\leq q_1.
\]

For the case where $p_1=q_1=p_0$ and thus when $w\in A_1\cap\RH_{(q_0/p_0)'}$, an analogous reasoning shows that for some $p_0<q_+$ we have $T\in S(p_0,q_+;w)$. Hence, it follows from Proposition~\ref{prop:sparseim} that $T$ is of weak type $(p_0,p_0)$ with respect to such weights.

Our main results deal with the cases $p_1=p_0$ where we establish quantitative bounds of $T$ in terms of the characteristic constants of the weight in the situations
\[
\begin{split}
T:L^p(w)\to L^p(w),&\quad w\in A_1\cap\RH_{(q_0/p)'},\\
T:L^{p_0}(w)\to L^{p_0,\infty}(w),&\quad w\in A_1\cap\RH_{(q_0/p_0)'}.
\end{split}
\]

\begin{proof}[Proof of Proposition \ref{prop:sparseim}]
By splitting into $3^n$ terms, we may assume without loss of generality that our sparse domination occurs in a single dyadic grid $\D^\alpha$ throughout our arguments.

Let $p_0<p<q_0$ and let $f\in L^p\cap\mathcal{D}$, $g\in L^{p'}\cap\mathcal{D}$. Then we can find a sparse collection $\Sp\subseteq\D^\alpha$ so that
\[
|\langle Tf,g\rangle|\lesssim \sum_{Q\in\Sp}\langle f\rangle_{p_0,Q}\langle g\rangle_{q_0',Q}|Q|\lesssim\sum_{Q\in\Sp}\essinf_Q\big(\M_{p_0}f\M_{q_0'}g\big)|E_Q|\lesssim\int\M_{p_0}f\M_{q_0'}g\,\mathrm{d}\mu.
\]
By using H\"{o}lder's inequality and by noting that $p>p_0$, $p'>q_0'$, it remains to observe that $\|\M_{p_0}f\|_p\lesssim\|f\|_p$ and $\|\M_{q_0'}g\|_{p'}\lesssim\|g\|_{p'}$. Hence, $T$ extends to a bounded operator in $L^p$.

For the second assertion we will use the equivalence
\begin{equation}\label{eq:weakchar}
\|T\|_{L^{p_0}\to L^{p_0,\infty}}\simeq\sup_{\|f\|_{p_0}=1}\sup_{\substack{E\subseteq\R^n\\0<|E|<\infty}}\inf_{\substack{E'\subseteq E\\ |E|\leq2|E'|}}\sup_{|h|\leq\chi_{E'}}|E|^{\frac{1}{p_0}-1}|\langle Tf,h\rangle|,
\end{equation}
with $f,h\in\mathcal{D}$, see \cite[Exercise 1.4.14]{grafakosclassic}. Given such an $f$ with $\|f\|_{p_0}=1$ and $E\subseteq\R^n$ of finite positive measure we define
\[
\Omega:=\Big\{\M^{\mathscr{B}}(|f|^{p_0})>K|E|^{-1}\Big\},\quad E':=E\backslash\Omega,
\]
where $K$ is chosen large enough to ensure that $|E|\leq 2|E'|$. Let $h\in\mathcal{D}$ with $|h|\leq\chi_{E'}$. Then we can find a sparse collection $\Sp\subseteq\D^\alpha$ such that
\begin{equation}\label{eq:inestsp}
|\langle Tf,h\rangle|\lesssim\sum_{Q\in\Sp}\langle f\rangle_{p_0,Q}\langle h\rangle_{q_0',Q}|Q|=\sum_{\substack{Q\in\Sp\\ Q\cap E'\neq\emptyset}}\langle f\rangle_{p_0,Q}\langle h\rangle_{q_0',Q}|Q|.
\end{equation}
We proceed by taking a Calder\'{o}n-Zygmund decomposition of $|f|^{p_0}\in L^1$. We can find a disjoint collection $\mathscr{P}\subseteq\D^\alpha$ of cubes so that $\Omega=\cup_{P\in\mathscr{P}}P$ and functions $g$, $(b_P)_{P\in\mathscr{P}}$ so that $|f|^{p_0}=g+\sum_{P\in\mathscr{P}}b_P$ and
\begin{align}\label{eq:caldweak1}
\supp b_P\subseteq P,\quad\int_P\!b_P\,\mathrm{d}\mu=0,\\\label{eq:caldweak2}
\|g\|_\infty\lesssim|E|^{-1},\quad\|g\|_{1}\lesssim 1
\end{align}
Noting that for all $P\in\mathscr{P}$ we have $P\cap E'=\emptyset$, the properties of the dyadic system imply that for any $Q\in\Sp$ with $Q\cap E'\neq\emptyset$ we have $P\subseteq Q$ whenever $P\cap Q\neq\emptyset$. But then by \eqref{eq:caldweak1} and arguments similar to the ones in the first part of the proof we have
\begin{equation}\label{eq:secestsp}
\sum_{\substack{Q\in\Sp\\ Q\cap E'\neq\emptyset}}\langle f\rangle_{p_0,Q}\langle h\rangle_{q_0',Q}|Q|=\sum_{\substack{Q\in\Sp\\ Q\cap E'\neq\emptyset}}\left(\frac{1}{|Q|}\int_Q\!g\,\mathrm{d}\mu\right)^{\frac{1}{p_0}}\langle h\rangle_{q_0',Q}|Q|\lesssim\||g|^{\frac{1}{p_0}}\|_q\|h\|_{q'}.
\end{equation}
Thus, by combining \eqref{eq:inestsp} and \eqref{eq:secestsp}, we find using \eqref{eq:caldweak2} that
\[
|\langle Tf,h\rangle|\lesssim\left(\|g\|^{1-\frac{p_0}{q}}_\infty\|g\|^{\frac{p_0}{q}}_1\right)^{\frac{1}{p_0}}\|h\|_\infty|E'|^{\frac{1}{q'}}\lesssim|E|^{\frac{1}{q}-\frac{1}{p_0}}|E|^{\frac{1}{q'}}=|E|^{1-\frac{1}{p_0}}.
\]
Hence, we may conclude from \eqref{eq:weakchar} that $\|T\|_{L^{p_0}\to L^{p_0,\infty}}<\infty$, finishing the proof.
\end{proof}
\begin{remark}\label{rem:czob}
The cancellation of the `bad` part $b$ in our proofs occurs because we are able to perform our Calder\'on-Zygmund decomposition in the same dyadic grid as where the sparse domination occurs, see Lemma \ref{lem:caldzyg1}. The usual Whitney decomposition argument that is used for Calder\'on-Zygmund decompositions in general doubling metric measure spaces, as can be found for example in \cite{CW71, stein93}, is not precise enough for this particular argument and we need to adapt the results so that they work with our dyadic grids.
\end{remark}

\section{Proofs of the main results}\label{sec:proofs}
Throughout these proofs we fix $\alpha\in\big\{0,\frac{1}{3},\frac{2}{3}\big\}^n$ and only consider cubes taken from the grid $\D^\alpha$. We also only consider the dyadic maximal operators $\M_p$ to be taken with respect to this grid to facilitate some of the arguments and for simpler constants in our estimates. Recall that $\mathcal{D}$ denotes a space of functions in $\R^n$ which has the property that it is dense in $L^p(w)$ for all $1\leq p<\infty$ and all weights $w\in A_\infty$.

As an analogue to \cite[Lemma 3.2]{lerner08} and \cite[Lemma 6.1]{hytonen13}, our main lemma is the following:
\begin{lemma}\label{lem:main}
Let $\Sp\subseteq\D^\alpha$ be $\eta$-sparse, and let $1\leq p_0<p<q_0\leq\infty$, $1<q<\infty$. Then there is a constant $c=c(\eta,n,\nu)>0$ so that
\[
\sum_{Q\in\Sp}\langle f\rangle_{p_0,Q}\langle g\rangle_{q_0',Q}|Q|\leq c c_p(q')^{\frac{1}{p'}}\|f\|_{L^p(\M_{q(q_0/p)'}w)}\|g\|_{L^{p'}(w^{1-p'})}
\]
for all $f,g\in\mathcal{D}$ and $w\in L^{q(q_0/p)'}_{loc}$, where $c_p$ is as in Theorem \ref{thm:main1}.
\end{lemma}
We point out that a similar type of result is established in \cite[Theorem B]{diplinio17}.
\begin{remark}\label{rem:unweightedopnorm}
In the unweighted case we note that
\[
\sum_{Q\in\Sp}\langle f\rangle_{p_0,Q}\langle g\rangle_{q_0',Q}|Q|\leq\eta^{-1}\|\M_{p_0}f\|_p\|\M_{q_0'}g\|_{p'}\leq\eta^{-1}\left[\bigg(\frac{p}{p_0}\bigg)'\right]^{\frac{1}{p_0}}\left[\bigg(\frac{p'}{q_0'}\bigg)'\right]^{\frac{1}{q_0'}}\|f\|_p\|g\|_{p'}.
\]
Thus, it appears that adding the weight accounts for the extra term $[(p_0'/p')']^{1/p_0}$ in the constant $c_p$, which depends on $p$ if and only if $p_0>1$. As a matter of fact, we shall see in the proof of Lemma \ref{lem:sparsecoifman} that this constant appears in an application of Kolmogorov's Lemma to the maximal operator. This extra term is what causes the additional terms in the quantitative bounds for $p_0>1$ in Theorem \ref{thm:main2} and at this moment we are unsure whether it can be removed or not.
\end{remark}
We break up the proof of the main lemma into a sequence of lemmata.
\begin{lemma}\label{lem:intermedmax}
For all $f,g\in\mathcal{D}$ and $0<\beta\leq 1$ we have
\[
\sum_{Q\in\Sp}\langle f\rangle_{p_0,Q}\langle g\rangle_{q_0',Q}|Q|\lesssim\int\M_{p_0}\left((\M_{q_0'} g)^{1-\beta}f\right)(\M_{q_0'} g)^\beta\,\mathrm{d}\mu.
\]
\end{lemma}
\begin{proof}
Note that for any $Q\in\D^\alpha$ we have
\[
\langle g\rangle_{q_0',Q}=\langle g\rangle_{q_0',Q}^\beta\langle g\rangle_{q_0',Q}^{1-\beta}\leq\langle g\rangle_{q_0',Q}^\beta\essinf_{Q}(\M_{q_0'} g)^{1-\beta}
\]
so that
\[
\langle f\rangle_{p_0,Q}\langle g\rangle_{q_0',Q}\leq\langle(\M_{q_0'} g )^{1-\beta}f\rangle_{p_0,Q}\langle g\rangle_{q_0',Q}^\beta\leq\essinf_Q \M_{p_0}\left((\M_{q_0'} g )^{1-\beta}f\right)(\M_{q_0'}g)^{\beta}.
\]
Hence, if $\Sp\subseteq\D^\alpha$ is sparse, we find that
\begin{align*}
\sum_{Q\in\Sp}\langle f\rangle_{p_0,Q}\langle g\rangle_{q_0',Q}|Q|&\lesssim\sum_{Q\in\Sp}\essinf_Q \M_{p_0}\left((\M_{q_0'} g )^{1-\beta}f\right)(\M_{q_0'}g)^{\beta}|E_Q|\\
&\leq\int\M_{p_0}\left((\M_{q_0'} g)^{1-\beta}f\right)(\M_{q_0'} g)^\beta\,\mathrm{d}\mu,
\end{align*}
as desired.
\end{proof}
\begin{lemma}\label{lem:sparsecoifman}
For all $1\leq q<\infty$, $w\in L^q_{loc}$, $p_0<p<q_0$, and $f,g\in\mathcal{D}$, we have
\[
\sum_{Q\in\Sp}\langle f\rangle_{p_0,Q}\langle g\rangle_{q_0',Q}|Q|\lesssim\tau_p\|f\|_{L^p(\M_q w)}\|\M_{q_0'} g\|_{L^{p'}((\M_qw)^{1-p'})},
\]
where
\[
\tau_p=\left[\left(\frac{p_0'}{p'}\right)'\left(\frac{p}{p_0}\right)'\right]^{\frac{1}{p_0}}.
\]
\end{lemma}
For the proof of this lemma we require two results on dyadic maximal operators. By the classical result of Fefferman and Stein \cite{fefstein71} we have
\begin{equation}\label{eq:classfefstein}
\|\M f\|_{L^{1,\infty}(w)}\leq\|f\|_{L^1(\M w)}
\end{equation}
and thus $\|\M f\|_{L^p(w)}\leq p'\|f\|_{L^p(\M w)}$ for $1<p<\infty$ by the Marcinkiewicz Interpolation Theorem. This implies that
\begin{equation}\label{eq:fefstein}
\|\M_q f\|_{L^p(w)}\leq\left[\left(\frac{p}{q}\right)'\right]^{\frac{1}{q}}\|f\|_{L^p(\M w)},\quad 1\leq q<p<\infty.
\end{equation}
Moreover, as a consequence of Kolmogorov's Lemma we have
\begin{equation}\label{eq:maxaone}
\M((\M f)^\delta)\lesssim\frac{(\M f)^\delta}{1-\delta},\quad 0<\delta<1
\end{equation}
for $f$ such that $\M f<\infty$. For this result we refer the reader to \cite[Proposition 2]{coifman80} or \cite[Theorem 9.2.7]{grafakosmodern}.
\begin{proof}
We will prove the stronger assertion
\[
\sum_{Q\in\Sp}\langle f\rangle_{p_0,Q}\langle g\rangle_{q_0',Q}|Q|\lesssim\left[\bigg(\frac{p_0'}{r}\bigg)'\bigg(\frac{\max(r,p')'}{p_0}\bigg)'\right]^{\frac{1}{p_0}}\|f\|_{L^p((\M w)^{(1-r)/(1-p')})}\|\M_{q_0'} g\|_{L^{p'}((\M w)^{1-r})},
\]
valid for all $1<r<p_0'$, generalizing a version of the result \cite[Theorem 1.7]{lerner10} and its proof in which the case $p_0=1$, $q_0=\infty$ is treated. The result of the lemma follows by taking $r=(p'-1)/q+1\in(1,p']$.

We set
\[
\beta:=\min\left(p'\frac{r-1+p_0'-1}{p_0'(r-1)+(p_0'-1)r},1\right)
\]
so that $0<\beta\leq 1$.

By Lemma \ref{lem:intermedmax} and by H\"{o}lder's Inequality we find that
\begin{equation}\label{eq:maxcoif}
\begin{split}
\sum_{Q\in\Sp}\langle f\rangle_{p_0,Q}\langle g\rangle_{q_0',Q}|Q|&\lesssim\int\M_{p_0}\left((\M_{q_0'} g)^{1-\beta}f\right)(\M_{q_0'} g)^\beta\,\mathrm{d}\mu\\
&\leq I\|\M_{q_0'} g\|^\beta_{L^{p'}((\M w)^{1-r})},
\end{split}
\end{equation}
where
\[
I=\left\|\M_{p_0}\left((\M_{q_0'} g )^{1-\beta}f\right)\right\|_{L^{\frac{p'}{p'-\beta}}((\M w)^{(1-r)\beta/(\beta-p')})}.
\]

We will consider two cases. First assume that
\[
p'\frac{r-1+p_0'-1}{p_0'(r-1)+(p'_0-1)r}\geq 1
\]
and $\beta=1$. Then
\[
(p'-1)(r-1+p_0'-1)\geq 2(p_0'-1)(r-1)
\]
so that
\[
\frac{1-r}{1-p'}\leq\frac{1}{2}\left(1+\frac{r-1}{p_0'-1}\right)<1
\]
by the assumption $r<p_0'$. Then it follows from \eqref{eq:fefstein} and \eqref{eq:maxaone} that
\begin{align*}
I&=\left\|\M_{p_0}f\right\|_{L^{p}((\M w)^{(1-r)/(1-p')})}\leq\left[\bigg(\frac{p}{p_0}\bigg)'\right]^{\frac{1}{p_0}}\|f\|_{L^{p}(\M((\M w)^{(1-r)/(1-p')})))}\\
&\lesssim\left(\frac{1}{1-\frac{1-r}{1-p'}}\right)^{\frac{1}{p}}\left[\bigg(\frac{p}{p_0}\bigg)'\right]^{\frac{1}{p_0}}\|f\|_{L^{p}((\M w)^{(1-r)/(1-p')})}\\
&\leq\left[2\bigg(\frac{p_0'}{r}\bigg)'\right]^{\frac{1}{p}}\left[\bigg(\frac{p}{p_0}\bigg)'\right]^{\frac{1}{p_0}}\|f\|_{L^{p}((\M w)^{(1-r)/(1-p')})},
\end{align*}
as desired.

For the second case we assume that
\[
p'\frac{r-1+p_0'-1}{p_0'(r-1)+(p_0'-1)r}<1\quad\text{and}\quad\beta=p'\frac{r-1+p_0'-1}{p_0'(r-1)+(p_0'-1)r}.
\]
Then, using $r<p_0'$, we note that
\[
\frac{p'}{p'-\beta}=\frac{p_0+r'}{2}>p_0\quad\text{and}\quad\frac{(1-r)\beta}{\beta-p'}=\frac{1}{2}\left(1+\frac{r-1}{p_0'-1}\right)<1.
\]
Hence, we may apply \eqref{eq:fefstein} and \eqref{eq:maxaone} so that
\begin{equation}\label{eq:maxcoif2}
\begin{split}
I&\lesssim\left(\frac{1}{1-\frac{(1-r)\beta}{\beta-p'}}\right)^{\frac{p'-\beta}{p'}}\left(\frac{\frac{p'}{p'-\beta}}{\frac{p'}{p'-\beta}-p_0}\right)^{\frac{1}{p_0}}\|(\M_{q_0'} g)^{1-\beta}f\|_{L^{\frac{p'}{p'-\beta}}((\M w)^{(1-r)\beta/(\beta-p')})}\\
&\leq\left[2\bigg(\frac{p_0'}{r}\bigg)'\right]^{\frac{p'-\beta}{p'}}\left[1+2\frac{p_0}{r'}\bigg(\frac{r'}{p_0}\bigg)'\right]^{\frac{1}{p_0}}\|(\M_{q_0'} g)^{1-\beta}f\|_{L^{\frac{p'}{p'-\beta}}((\M w)^{(1-r)\beta/(\beta-p')})}.
\end{split}
\end{equation}

By H\"{o}lder's Inequality we find that
\begin{align*}
\|(\M_{q_0'} g )^{1-\beta}f\|^{\frac{p'}{p'-\beta}}_{L^{\frac{p'}{p'-\beta}}((\M w)^{(1-r)\beta/(\beta-p')})}&=
\int(\M_{q_0'} g)^{\frac{(1-\beta)p'}{p'-\beta}}(\M w)^{\frac{(1-r)(\beta-1)}{\beta-p'}} |f|^{\frac{p'}{p'-\beta}}(\M w)^{\frac{1-r}{\beta-p'}}\,\mathrm{d}\mu\\
&\leq\|\M_{q_0'}g\|_{L^{p'}((\M w)^{1-r})}^{\frac{p'(1-\beta)}{p'-\beta}}\|f\|^{\frac{p'}{p'-\beta}}_{L^p((\M w)^{(1-r)/(1-p')})}.
\end{align*}
Hence, by \eqref{eq:maxcoif2} we have
\[
I\lesssim\left[\bigg(\frac{p_0'}{r}\bigg)'\bigg(\frac{r'}{p_0}\bigg)'\right]^{\frac{1}{p_0}}\|f\|_{L^p((\M w)^{(1-r)/(1-p')})}\|\M_{q_0'} g\|_{L^{p'}((\M w)^{1-r})}^{1-\beta}.
\]
Thus, the result follows from \eqref{eq:maxcoif}.

By combining the two cases, the assertion follows.
\end{proof}
For the proof of Lemma \ref{lem:main} we will use a result that can be found in \cite[p.~8]{lerner08} which states that
\begin{equation}\label{eq:maxwduo}
\|\M f\|_{L^{p'}((\M_q w)^{1-p'})}\leq \left(\frac{pq-1}{q-1}\right)^{1-\frac{1}{pq}}\|f\|_{L^{p'}(w^{1-p'})}
\end{equation}
for $1<p,q<\infty$.
\begin{proof}[Proof of Lemma \ref{lem:main}]
Setting $v:=w^{(q_0/p)'}$, it follows from \eqref{eq:maxwduo} that
\begin{equation}\label{eq:appcoif2}
\begin{split}
\|\M_{q_0'}g\|_{L^{p'}\left(\left(\M_{q(q_0/p)'}w\right)^{1-p'}\right)}&=\|\M(|g|^{q_0'})\|^{\frac{1}{q_0'}}_{L^{p'/q_0'}\big(\left(\M_q v\right)^{1-p'/q_0'}\big)}\\
&\leq\left(\frac{\left(\frac{p'}{q_0'}\right)'q-1}{q-1}\right)^{\frac{1}{q_0'}-\frac{1}{q_0'\left(\frac{p'}{q_0'}\right)'q}}\||g|^{q_0'}\|^{\frac{1}{q_0'}}_{L^{\frac{p'}{q_0'}}\big(v^{1-p'/q_0'}\big)}\\
&\leq\left[\bigg(\frac{p'}{q_0'}\bigg)'\right]^{\frac{1}{q_0'}}(q')^{\frac{1}{p'}+\frac{1}{q'}}\|g\|_{L^{p'}(w^{1-p'})},
\end{split}
\end{equation}
where we used that
\[
\frac{1}{q_0'}-\frac{1}{q_0'\left(\frac{p'}{q_0'}\right)'q}=\frac{1}{p'}+\frac{1}{q_0'\left(\frac{p'}{q_0'}\right)'q'}\leq \frac{1}{p'}+\frac{1}{q'}.
\]
By maximizing the function $t\mapsto t^{1/t}$ for $t\geq 1$, we note that $(q')^{1/q'}\leq e^{1/e}$. Hence, by combining Lemma \ref{lem:sparsecoifman} and \eqref{eq:appcoif2}, the result follows.
\end{proof}
\begin{proof}[Proof of Theorem \ref{thm:main1}]
Set $v:=w^{(q_0/p)'}\in A_1$ and let $\kappa$ be the constant from Proposition \ref{prop:weightprop}(iii). Setting $q=1+1/(\kappa[v]_{A_\infty})$, we have $q'\simeq[v]_{A_\infty}$ and
\[
\M_{q(q_0/p)'}w=(\M_q v)^{\frac{1}{(q_0/p)'}}\lesssim[v]^{\frac{1}{(q_0/p)'}}_{A_1}w.
\]
Hence, from Lemma \ref{lem:main} it follows that
\[
|\langle Tf,g\rangle|\lesssim c_p[v]^{\frac{1}{p'}}_{A_\infty}[v]^{^{\frac{1}{p(q_0/p)'}}}_{A_1}\|f\|_{L^p(w)}\|g\|_{L^{p'}(w^{1-p'})},
\]
proving the result.
\end{proof}
\begin{proof}[Proof of Theorem \ref{thm:main2}]
The proof uses arguments similar to the ones presented in the proof of Proposition \ref{prop:sparseim}. We use the equivalence
\begin{equation}\label{eq:weakwt}
\|T\|_{L^{p_0}(w)\to L^{p_0,\infty}(w)}\simeq\sup_{\|f\|_{L^{p_0}(w)}=1}\sup_{\substack{E\subseteq\R^n\\ 0<w(E)<\infty}}\inf_{\substack{E'\subseteq E\\ w(E)\leq 2 w(E')}}\sup_{|h|\leq\chi_{E'}} w(E)^{\frac{1}{p_0}-1}\left|\int(Tf)hw\,\mathrm{d}\mu\right|.
\end{equation}
Fixing a function $f\in\mathcal{D}$ with $\|f\|_{L^{p_0}(w)}=1$ and a measurable set $E$, we set $\Omega:=\{\M^{\mathscr{B}}(|f|^{p_0})> 2c[w]_{A_1} w(E)^{-1}\}$, where $\M^{\mathscr{B}}$ denotes the uncentered maximal operator with respect to all balls $B\subseteq\R^n$ and where $c=c(n,\nu)>0$ is the constant appearing in the inequality $\|\M^{\mathscr{B}}\phi\|_{L^{1,\infty}(w)}\leq c[w]_{A_1}\|\phi\|_{L^1(w)}$, which is a consequence of \eqref{eq:classfefstein}. We have
\begin{equation}\label{eq:needthis}
w(\Omega)\leq\frac{c[w]_{A_1}w(E)}{2c[w]_{A_1}}\int\!|f|^{p_0}w\,\mathrm{d}\mu=\frac{w(E)}{2}
\end{equation}
and thus, setting $E':=E\backslash\Omega$, we have $w(E')\geq w(E)-w(\Omega)\geq w(E)/2$.

By applying Lemma \ref{lem:caldzyg1} with $|f|^{p_0}\in L^1$, we obtain a disjoint collection $\mathscr{P}\subseteq\D^\alpha$ of cubes so that $\Omega=\cup_{P\in\mathscr{P}}P$ and functions $g$, $b$ so that $|f|^{p_0}=g+b$, where
\[
g=|f|^{p_0}\chi_{\Omega^c}+\sum_{P\in\mathscr{P}}\langle|f|^{p_0}\rangle_{1,P}\chi_P
\]
and
\[
\|g\|_\infty\lesssim\frac{[w]_{A_1}}{w(E)}.
\]
Picking a function $h$ satisfying $|h|\leq\chi_{E'}$ and $hw\in\mathcal{D}$, we apply the sparse domination property to the pair $f$, $hw$ to find a sparse collection $\Sp\subseteq\D^\alpha$ so that, by using Lemma \ref{lem:main} with the weight $w\chi_{E'}$, for all $p_0<p<q_0$ and $1<q<\infty$ we have
\begin{equation}\label{eq:lempropest}
\begin{split}
\left|\int(Tf)hw\,\mathrm{d}\mu\right|&=|\langle Tf,hw\rangle|\leq\sum_{\substack{Q\in\Sp\\ Q\cap E'\neq\emptyset}}\left(\frac{1}{|Q|}\int_Q\!g\,\mathrm{d}\mu\right)^{\frac{1}{p_0}}\langle hw\rangle_{q_0',Q}|Q|\\
&\lesssim c_p(q')^{\frac{1}{p'}}\||g|^{\frac{1}{p_0}}\|_{L^p(\M_{q(q_0/p)'}(w\chi_{E'}))}\|hw\|_{L^{p'}(w^{1-p'})}\\
&\lesssim c_p(q')^{\frac{1}{p'}}[w]_{A_1}^{\frac{1}{p_0}-\frac{1}{p}}w(E)^{\frac{1}{p}-\frac{1}{p_0}}\|g\|^{\frac{1}{p}}_{L^1(\M_{q(q_0/p)'}(w\chi_{E'}))}w(E')^{\frac{1}{p'}}.
\end{split}
\end{equation}
Note here that we have used the fact that the terms involving $b$ cancel in the exact same way as they do in the proof of Proposition \ref{prop:sparseim}.

Similar to what is done in \cite{perez94, lerner08, hytonen13}, we deal with the term involving $g$ as follows: We remark that for a cube $P\in\D^\alpha$ we have
\begin{equation}\label{eq:maxconst}
\M(\phi\chi_{P^c})(x)=\essinf_P\M(\phi\chi_{P^c})\quad\text{for all $x\in P$.}
\end{equation}
Indeed, let $x,y\in P$ and let $R\in\D^\alpha$ so that $x\in R$. Then either $R\subseteq P$ or $P\subseteq R$. In the first case we have $\langle\phi\chi_{P^c}\rangle_{1,R}=0$ while in the second case we have $y\in R$ and thus $\langle\phi\chi_{P^c}\rangle_{1,R}\leq\M(\phi\chi_{P^c})(y)$. Thus, we may conclude that $\M(\phi\chi_{P^c})(x)\leq\M(\phi\chi_{P^c})(y)$, proving \eqref{eq:maxconst} by symmetry. Using this result, we find, since $E'\subseteq P^c$ for all $P\in\mathscr{P}$, that
\begin{align*}
\int_\Omega\!|g|\M_{q(q_0/p)'}(w\chi_{E'})\,\mathrm{d}\mu&\leq\sum_{P\in\mathscr{P}}\essinf_P\M_{q(q_0/p)'}(w\chi_{P^c})\int_P\!|f|^{p_0}\,\mathrm{d}\mu\\
&\leq \int_\Omega\!|f|^{p_0}\M_{q(q_0/p)'}w\,\mathrm{d}\mu.
\end{align*}
Since $g=|f|^{p_0}$ on $\Omega^c$, we conclude that
\begin{equation}\label{eq:gtof}
\|g\|_{L^1(\M_{q(q_0/p)'}(w\chi_{E'}))}\leq\|f\|^{p_0}_{L^{p_0}(\M_{q(q_0/p)'}w)}.
\end{equation}

We first assume that $q_0<\infty$. We set $v:=w^{(q_0/p_0)'}\in A_1$ and choose
\[
p=p_0+\frac{q_0-p_0}{1+2\kappa[v]_{A_\infty}},\quad q=\frac{2+2\kappa[v]_{A_\infty}}{1+2\kappa[v]_{A_\infty}}.
\]
Then we have $q'=2+2\kappa[v]_{A_\infty}\simeq[v]_{A_\infty}$, and
\[
q\bigg(\frac{q_0}{p}\bigg)'\bigg/\bigg(\frac{q_0}{p_0}\bigg)'=1+\frac{1}{\kappa[v]_{A_\infty}}
\]
so that it follows from Proposition \ref{prop:weightprop}(iii) that
\[
\M_{q(q_0/p)'}w\lesssim[v]_{A_1}^{\frac{1}{(q_0/p_0)'}}w.
\]
Thus, it follows from \eqref{eq:lempropest}, \eqref{eq:gtof}, and Proposition \ref{prop:weightprop}(ii) that
\begin{equation}\label{eq:almostfin}
w(E)^{\frac{1}{p_0}-1}\left|\int(Tf)hw\,\mathrm{d}\mu\right|\lesssim c_p[v]^{\frac{1}{p'}}_{A_\infty}[w]^{\frac{1}{p_0}-\frac{1}{p}}_{A_1}[v]^{\frac{1}{p(q_0/p_0)'}}_{A_1}\leq c_p[v]^{\frac{1}{p'}}_{A_\infty}\left([w]_{A_1}[w]_{\RH_{(q_0/p_0)'}}\right)^{\frac{1}{p_0}}.
\end{equation}
Next, we note that
\[
\frac{1}{p'}=\frac{q_0-1+2\kappa[v]_{A_\infty}(p_0-1)}{q_0+2\kappa[v]_{A_\infty}p_0}\leq\frac{q_0-1}{2\kappa[v]_{A_\infty}p_0}+\frac{1}{p_0'}
\]
and thus
\[
[v]_{A_\infty}^{\frac{1}{p'}}\lesssim[v]_{A_\infty}^{\frac{1}{p_0'}}.
\]
Moreover, we compute
\begin{align*}
c_p&=\left[\frac{2(q_0-1)(q_0+2\kappa[v]_{A_\infty}p_0)}{2\kappa[v]_{A_\infty}(q_0-p_0)}\right]^{\frac{1}{q_0'}}\left[\frac{p_0(q_0+2\kappa[v]_{A_\infty})(q_0-1+2\kappa[v]_{A_\infty}(p_0-1))}{(q_0-p_0)^2}\right]^{\frac{1}{p_0}}\\
&\lesssim[v]_{A_\infty}^{\frac{1}{p_0}}\left[1+(p_0-1)[v]_{A_\infty}\right]^{\frac{1}{p_0}}.
\end{align*}
Hence, it follows from \eqref{eq:weakwt} and \eqref{eq:almostfin} that
\[
\|T\|_{L^{p_0}(w)\to L^{p_0,\infty}(w)}\lesssim[v]_{A_\infty}\left[1+(p_0-1)[v]_{A_\infty}\right]^{\frac{1}{p_0}}\left([w]_{A_1}[w]_{\RH_{(q_0/p_0)'}}\right)^{\frac{1}{p_0}}.
\]
The result follows by considering the cases $p_0=1$ and $p_0>1$ separately.

Now we assume that $q_0=\infty$. Taking $q=1+1/(\kappa[w]_{A_\infty})$ we have $q'\simeq[w]_{A_\infty}$. Thus, from \eqref{eq:lempropest} and Proposition \ref{prop:weightprop}(iii) we obtain
\begin{equation}\label{eq:wealmfin}
w(E)^{\frac{1}{p_0}-1}\left|\int(Tf)hw\,\mathrm{d}\mu\right|\lesssim c_p[w]^{\frac{1}{p'}}_{A_\infty}[w]^{\frac{1}{p_0}}_{A_1}
\end{equation}
for all $p_0<p<\infty$. Choosing $p=p_0+1/(\log(e+[w]_{A_\infty}))$, we have
\[
\frac{1}{p'}=\frac{1+(p_0-1)\log(e+[w]_{A_\infty})}{1+p_0\log(e+[w]_{A_\infty})}\leq\frac{1}{p_0\log(e+[w]_{A_\infty})}+\frac{1}{p_0'}
\]
so that
\begin{equation}\label{eq:wafend}
[w]_{A_\infty}^{\frac{1}{p'}}\lesssim[w]_{A_\infty}^{\frac{1}{p_0'}}.
\end{equation}
Moreover, we compute
\begin{align*}
c_p&=p\left[p_0(1+(p_0-1)\log(e+[w]_{A_\infty}))(1+p_0\log(e+[w]_{A_\infty}))\right]^{\frac{1}{p_0}}\\
&\lesssim\big[1+(p_0-1)\log(e+[w]_{A_\infty})\big]^{\frac{1}{p_0}}\log(e+[w]_{A_\infty})^{\frac{1}{p_0}}.
\end{align*}
Hence, by \eqref{eq:weakwt}, \eqref{eq:wealmfin}, and \eqref{eq:wafend}, we conclude that
\[
\|T\|_{L^{p_0}(w)\to L^{p_0,\infty}(w)}\lesssim[w]^{\frac{1}{p_0}}_{A_1}[w]^{\frac{1}{p_0'}}_{A_\infty}\big[1+(p_0-1)\log(e+[w]_{A_\infty})\big]^{\frac{1}{p_0}}\log(e+[w]_{A_\infty})^{\frac{1}{p_0}}.
\]
By considering the cases $p_0=1$ and $p_0>1$ separately, the desired result follows.
\end{proof}
\begin{proof}[Proof of Theorem \ref{thm:maindual}]
We use the equivalence
\begin{equation}\label{eq:weakwt2}
\left\|\frac{T^\ast f}{w^{\frac{1}{q_0'}}}\right\|_{L^{q_0',\infty}(w)}\simeq\sup_{\substack{E\subseteq\R^n\\ 0<w(E)<\infty}}\inf_{\substack{E'\subseteq E\\ w(E)\leq 2 w(E')}}\sup_{|h|\leq\chi_{E'}} w(E)^{-\frac{1}{q_0}}\left|\left\langle\frac{T^\ast f}{w^{\frac{1}{q_0'}}},hw\right\rangle\right|.
\end{equation}
Let $f\in\mathcal{D}$ with $\|f\|_{q_0'}=1$ and let $E\subseteq\R^n$ with $0<w(E)<\infty$. We denote by $\M_w^{\mathscr{B}}$ the uncentered maximal operator over balls with respect to the measure $w\,\mathrm{d}\mu$. Then we define
\[
\Omega:=\Big\{\M_w^{\mathscr{B}}\Big(\frac{|f|^{q_0'}}{w}\Big)> 2cw(E)^{-1}\Big\},
\]
where $c=c(n,\nu)>0$ is the constant appearing in the inequality $\|\M_w^{\mathscr{B}}\phi\|_{L^{1,\infty}(w)}\leq c\|\phi\|_{L^1(w)}$. We have
\[
w(\Omega)\leq\frac{cw(E)}{2c}\int\!\frac{|f|^{q_0'}}{w}w\,\mathrm{d}\mu=\frac{w(E)}{2}
\]
which, setting $E':=E\backslash\Omega$, implies that $w(E')\geq w(E)-w(\Omega)\geq w(E)/2$.

By applying the Whitney Decomposition Theorem to $\Omega$, see Theorem \ref{thm:whitney} below, we obtain a disjoint collection $\mathscr{P}\subseteq\D^\alpha$ of cubes so that $\Omega=\cup_{P\in\mathscr{P}}P$ with the property that for each $P\in\mathscr{P}$ there exists a ball $B(P)$ containing $P$ so that $B(P)\cap\Omega^c\neq\emptyset$ and $|B(P)|\lesssim |P|$, where the implicit constant depends only on $n$ and $\nu$, see also the proof of Lemma \ref{lem:caldzyg1}. Moreover, we obtain functions $g$, $b$ so that $|f|^{q_0'}=g+b$, where
\[
g=|f|^{q_0'}\chi_{\Omega^c}+\sum_{P\in\mathscr{P}}\langle|f|^{q_0'}\rangle_{1,P}\chi_P.
\]
Next, we pick a function $h$ satisfying $|h|\leq\chi_{E'}$ and $hw^{1/q_0}\in\mathcal{D}$, and fix a $p_0<p<q_0$ to be chosen later. We apply the sparse domination property to the pair $hw^{1/q_0}$, $f$ to find a sparse collection $\Sp\subseteq\D^\alpha$ so that, by applying Lemma \ref{lem:main} with the weight $w^{1/(q_0/p)'}$, we find that for all $1<q<\infty$ we have
\begin{equation}\label{eq:lempropestv}
\begin{split}
\left|\left\langle\frac{T^\ast f}{w^{\frac{1}{q_0'}}},hw\right\rangle\right|&=\left|\left\langle f,T\left(hw^{\frac{1}{q_0}}\right)\right\rangle\right|\leq\sum_{\substack{Q\in\Sp\\ Q\cap E'\neq\emptyset}}\left\langle hw^{\frac{1}{q_0}}\right\rangle_{p_0,Q}\left(\frac{1}{|Q|}\int_Q\!g\,\mathrm{d}\mu\right)^{\frac{1}{q_0'}}|Q|\\
&\lesssim c_p(q')^{\frac{1}{p'}}\|hw^{\frac{1}{q_0}}\|_{L^{p}((\M_{q}w)^{\frac{1}{(q_0/p)'}})}\||g|^{\frac{1}{q'_0}}\|_{L^{p'}(w^{\frac{1-p'}{(q_0/p)'}})}\\
&=c_p(q')^{\frac{1}{p'}}\|h w^{\frac{1}{q_0}}\|_{L^{p}((\M_{q}w)^{\frac{1}{(q_0/p)'}})}\left(\int_{\Omega^c}\!|f|^{p'}w^{\frac{1-p'}{(q_0/p)'}}\,\mathrm{d}\mu+\sum_{P\in\mathscr{P}}\langle f\rangle_{q_0',P}^{p'}\int_P\!w^{\frac{1-p'}{(q_0/p)'}}\,\mathrm{d}\mu\right)^{\frac{1}{p'}},
\end{split}
\end{equation}
where the terms involving $b$ cancel in the same way as before.

Choosing $q=1+1/(\kappa[w]_{A_\infty})$ so that $q'\simeq[w]_{A_\infty}$, it follows from Proposition \ref{prop:weightprop}(iii) that
\begin{equation}\label{eq:hv1}
\begin{split}
(q')^{\frac{1}{p'}}\|h w^{\frac{1}{q_0}}\|_{L^{p}((\M_{q}w)^{\frac{1}{(q_0/p)'}})}&\lesssim[w]^{\frac{1}{p'}}_{A_\infty}[w]_{A_1}^{\frac{1}{p(q_0/p)'}}\left(\int|h|^p w^{\frac{p}{q_0}}w^{\frac{1}{(q_0/p)'}}\,\mathrm{d}\mu\right)^{\frac{1}{p}}\\
&\leq[w]^{\frac{1}{p'}}_{A_\infty}[w]_{A_1}^{\frac{1}{p(q_0/p)'}}w(E')^{\frac{1}{p}}.
\end{split}
\end{equation}

Next, since $|f|\lesssim w(E)^{-1/q_0'}w^{1/q_0'}$ in $\Omega^c$, we have
\begin{equation}\label{eq:hv2}
\int_{\Omega^c}\!|f|^{p'}w^{\frac{1-p'}{(q_0/p)'}}\,\mathrm{d}\mu\leq w(E)^{\frac{q_0'-p'}{q_0'}}\int_{\Omega^c}\!|f|^{q_0'}w^{\frac{p'-q_0'}{q_0'}}w^{\frac{q_0'-p'}{q_0'}}\,\mathrm{d}\mu\leq w(E)^{\frac{q_0'-p'}{q_0'}}.
\end{equation}
Furthermore, fixing a $P\in\mathscr{P}$ and $x\in B(P)\cap\Omega^c$, we have
\begin{align*}
\langle f\rangle_{q_0',P}^{p'-q_0'}&\lesssim\langle f\rangle_{q_0',B(P)}^{p'-q_0'}\leq\left[\M^{\mathscr{B}}_w\Big(\frac{|f|^{q_0'}}{w}\Big)(x)\right]^{\frac{p'-q_0'}{q_0'}}\langle w\rangle_{1,B(P)}^{\frac{p'-q_0'}{q_0'}}\\
&\lesssim w(E)^{\frac{q_0'-p'}{q_0'}}\langle w\rangle_{1,B(P)}^{\frac{p'-q_0'}{q_0'}}
\end{align*}
and
\[
\Big\langle w^{\frac{q_0'-p'}{q_0'}}\Big\rangle_{1,P}\lesssim[w]_{A_1}^{\frac{p'-q_0'}{q_0'}}\Big\langle(\M^{\mathscr{B}}w)^{\frac{q_0'-p'}{q_0'}}\Big\rangle_{1,B(P)}\leq[w]_{A_1}^{\frac{p'-q_0'}{q_0'}}\langle w\rangle_{1,B(P)}^{\frac{q_0'-p'}{q_0'}}
\]
so that
\begin{equation}\label{eq:hv3}
\sum_{P\in\mathscr{P}}\langle f\rangle_{q_0',P}^{p'}\int_P\!w^{\frac{1-p'}{(q_0/p)'}}\,\mathrm{d}\mu\lesssim[w]^{\frac{p'-q_0'}{q_0'}}_{A_1}w(E)^{\frac{q_0'-p'}{q_0'}}\sum_{P\in\mathscr{P}}\langle|f|^{q_0'}\rangle_{1,P}|P|\leq[w]^{\frac{p'-q_0'}{q_0'}}_{A_1}w(E)^{\frac{q_0'-p'}{q_0'}}.
\end{equation}
Thus, by combining \eqref{eq:hv1},\eqref{eq:hv2}, and \eqref{eq:hv3} with \eqref{eq:lempropestv}, we conclude that
\begin{equation}\label{eq:hvfin}
\begin{split}
\left|\left\langle\frac{T^\ast f}{w^{\frac{1}{q_0'}}},hw\right\rangle\right|&\lesssim c_p[w]_{A_\infty}^{\frac{1}{p'}}[w]_{A_1}^{\frac{1}{p(q_0/p)'}}[w]_{A_1}^{\frac{p'-q_0'}{p'q_0'}}w(E')^{\frac{1}{p}}w(E)^{\frac{q_0'-p'}{p'q_0'}}\\
&\leq c_p[w]_{A_\infty}^{\frac{1}{q_0'}}[w]_{A_1}^{\frac{2}{p(q_0/p)'}}w(E)^{\frac{1}{q_0}}.
\end{split}
\end{equation}
By writing $L:=\log(e+[w]_{A_1})$ and choosing
\[
p=p_0\frac{q_0}{q_0+L}+q_0\frac{L}{q_0+L}\in(p_0,q_0)
\]
we have
\[
[w]_{A_1}^{\frac{2}{p(q_0/p)'}}=[w]_{A_1}^{\frac{2}{(q_0/p_0)'(p_0+L)}}\leq e^{2/e}
\]
and
\[
c_p=\left[\frac{q_0-1}{q_0-p_0}(p_0+L)\right]^{\frac{1}{q_0'}}\left[p_0\bigg(\frac{q_0}{p_0}\bigg)'\frac{p_0+L}{L}\frac{(p_0-1)\left(\frac{q_0}{p_0}\right)'+\frac{q_0-1}{q_0-p_0}L}{L}\right]^{\frac{1}{p_0}}\lesssim L^{\frac{1}{q_0'}}.
\]
Thus, by \eqref{eq:weakwt2} and \eqref{eq:hvfin} we have
\[
\left\|\frac{T^\ast f}{w^{\frac{1}{q_0'}}}\right\|_{L^{q_0',\infty}(w)}\lesssim\left([w]_{A_\infty}L\right)^{\frac{1}{q_0'}},
\]
as desired.
\end{proof}
\section{Extensions of the results to spaces of homogeneous type}\label{sec:ext}
This section is dedicated to extending our main results to spaces of homogeneous type $(X,d,\mu)$. Here $X$ is a set equipped with a quasimetric $d$, i.e., a mapping satisfying the usual properties of a metric except for the triangle inequality, which is replaced by the estimate
\[
d(x,y)\leq A(d(x,z)+d(z,y))
\]
for a constant $A\geq 1$, and $\mu$ is a Borel measure on $X$ satisfying the doubling property, i.e., there is a $C>0$ such that
\[
\mu(B(x;2r))\leq C\mu(B(x;r))
\]
for all $x\in X$, $r>0$. Taking the smallest such $C$ we set $\nu:=\log_2 C$. Furthermore, we write $|E|:=\mu(E)$ for all Borel sets $E\subseteq X$. The doubling property implies that for $x\in X$ and $R\geq r>0$ we have
\begin{equation}\label{eq:doublingpropnu}
|B(x;R)|\leq C\left(\frac{R}{r}\right)^\nu|B(x;r)|
\end{equation}
In turn, this implies that if $y\in B(x;R)$ for $x\in X$, then for $0<r\leq 2AR$ we have
\begin{equation}\label{eq:doublingpropnutwo}
|B(x;R)|\leq C\left(\frac{2AR}{r}\right)^\nu|B(y;r)|.
\end{equation}
We make the additional assumption that $0<|B|<\infty$ for all balls $B\subseteq X$. This property ensures that $X$ is separable \cite[Proposition 1.6]{bjorn11}.

Finally, we make the assumption that Lebesgue's Differentiation Theorem holds. This holds, for example, when $X$ is a domain in $\R^n$. Indeed, more generally, if $A=1$ (that is, $(X,d)$ is a metric space) and $\mu$ is an inner regular Borel outer measure, then Lebesgue's Differentation Theorem holds, see \cite[Section 14]{hajlasz00}. This assumption is used for the $L^\infty$ bound on the good part in our Calder\'on-Zygmund decompositions.

We will consider the situations where $X$ is unbounded and where $X$ is bounded separately, the latter situation being simpler. To facilitate this, we impose that the underlying quasimetric space $(X,d)$ has exactly one of the following properties:
\begin{enumerate}[(I)]
\item There is a constant $\gamma>0$ so that
\begin{equation}\label{eq:gamma}
\diam(B(x;r))\geq\gamma r
\end{equation}
for all $x\in X$, $r>0$;
\item $\diam X<\infty$.
\end{enumerate}
We note that property (I) and property (II) are mutually exclusive, since (I) implies that $X$ is unbounded. The extra assumption for the unbounded case is not too restrictive in the sense that the unbounded spaces in our applications usually do satisfy property (I). We point out that when $(X,d)$ is a connected metric space, then it satisfies either (I) or (II):
\begin{proposition}
Suppose $X$ is metric, connected, and unbounded. Then (I) holds with $\gamma=1$.
\end{proposition}
\begin{proof}
Let $r>\varepsilon>0$. The assumptions on $X$ imply that $X\neq\overline{B(x;r-\varepsilon)}\cup B(x;r)^c$ and thus we can pick $y\in B(x;r)\backslash\overline{B(x;r-\varepsilon)}$ so that $\diam(B(x;r))\geq d(x,y)\geq r-\varepsilon$, proving the result.
\end{proof}
A non-connected example where (I) holds with $\gamma=1/2$ is the subset $(-\infty,0)\cup (1,2)$ of the real line. An example where (I) fails is any metric space that has an isolated point.

We will use the following definition of a dyadic system in $X$.
\begin{definition}
Let $0<c_0\leq C_0<\infty$ and $0<\delta<1$. If for each $k\in\Z$ we have a pairwise disjoint collection $\D_k=(Q^k_j)_{j\in J_k}$ of measurable subsets of $X$ and a collection of points $(z^k_j)_{j\in J_k}$, then we call $(\D_k)_{k\in\Z}$ a \textit{dyadic system} in $X$ with parameters $c_0$, $C_0$, $\delta$, if it satisfies the following properties:
\begin{enumerate}[(i)]
\item for all $k\in\Z$ we have
\[
X=\bigcup_{j\in J_k}Q_j^k;
\]
\item for $l\geq k$, if $Q\in\D_l$ and $Q'\in\D_k$, we have that either $Q\cap Q'=\emptyset$ or $Q\subseteq Q'$;
\item for each $k\in\Z$ and $j\in J_k$ we have
\[
B(z^k_j;c_0\delta^k)\subseteq Q^k_j\subseteq B(z^k_j;C_0\delta^k);
\]
\item for $l\geq k$, if $Q^l_{j'}\subseteq Q^k_j$, then $B(z^l_{j'};C_0\delta^k)\subseteq B(z^k_j;C_0\delta^k)$.
\end{enumerate}
\end{definition}
The elements of a dyadic system are called cubes. We call $z^k_j$ the \textit{center} of $Q^k_j$. If $Q\in\D_k$, then we call the unique cube $Q'\in\D_{k-1}$ so that $Q\subseteq Q'$ the \textit{parent} of $Q$. Furthermore, we say that $Q$ is a \textit{child} of $Q'$. Note that it is possible that for a cube $Q$ there exists more than one $k\in\Z$ so that $Q\in\D_k$. Hence, when speaking of a child or the parent of $Q$, this should be with respect to a specific $k\in\Z$ where $Q\in\D_k$ to avoid ambiguity.

For a detailed discussion on the construction of dyadic systems and for the following theorem we refer the reader to \cite{hytonen12} and references therein. 

\begin{theorem}\label{thm:dyadicgrids}
There exist $0<c_0<C_0<\infty$, $0<\delta<1$, $\rho>0$ and a positive integer $K$, so that there are dyadic system $\D^1,\ldots,\D^K$ in $X$ with parameters $c_0$, $C_0$, $\delta$ so that for each $x\in X$ and $r>0$ there exists an $\alpha\in\{1,\ldots, K\}$ and $Q\in\D^\alpha$ so that
\[
B(x;r)\subseteq Q\quad\text{and}\quad \diam(Q)\leq\rho r.
\]
\end{theorem}
Writing $\D:=\cup_{\alpha=1}^K\D^\alpha$, one defines the respective notions for weight classes accordingly. Likewise, we say that a collection $\Sp\subseteq\D$ is called \emph{$\eta$-sparse} for $0<\eta\leq1$ if for each $\alpha\in\big\{1,\ldots, K\big\}$ there is a pairwise disjoint collection $(E_Q)_{Q\in\Sp\cap\D^\alpha}$ of measurable sets so that $E_Q\subseteq Q$ and $|Q|\leq\eta^{-1}|E_Q|$.

For our main results we require that the Calder\'on-Zygmund decompositions we take are adapted to the dyadic grids obtained from this theorem. The standard Calder\'on-Zygmund decomposition as found in \cite{CW71} is not precise enough for these purposes, see also Remark \ref{rem:czob}.

For $1\leq p_0<q_0\leq\infty$ we may define the class $S(p_0,q_0)$ as the class of those operators $T$ that satisfy the property that there is a constant $c>0$ and an $0<\eta\leq 1$ so that for each pair of functions $f$, $g$ in an appropriately large class of functions on $X$ there is an $\eta$-sparse collection $\Sp\subseteq\D$ so that
\[
|\langle Tf,g\rangle|\leq c\sum_{Q\in\Sp}\langle f\rangle_{p_0,Q}\langle g\rangle_{q_0',Q}|Q|.
\]
The remainder of this section will be dedicated to proving the following result:
\begin{theorem}\label{thm:thingswork}
Let $1\leq p_0<q_0\leq\infty$ and suppose that $(X,d)$ satisfies either property (I) or property (II). Then for $T\in S(p_0,q_0)$, the results of Theorem \ref{thm:main1} and Theorem \ref{thm:main2} remain true, where the dependence on $n$ of the constants changes to dependence on the parameters of the dyadic system (and also $\gamma$ in the case (I)). Similarly, the results of Theorem \ref{thm:maindual} remain true in the case that property (I) is satisfied.
\end{theorem}
The main difficulty arises when one wants to take Calder\'on-Zygmund decompositions. We remark that in the cases (I) and (II) one can use the standard maximal cube arguments and localization arguments respectively to conclude that our dyadic maximal operators satisfy the usual weak and strong boundedness results. The Lemmata in Section \ref{sec:proofs} all follow in the more general setting in the same way as they have been presented, where we replace the set of test functions $\mathcal{D}$ by another appropriate class of functions that is dense in $L^p(w)$ for all $1\leq p<\infty$, $w\in A_\infty$ such as the linear span of the indicator functions functions over the balls in $X$.

From now on we consider a fixed dyadic system $\D^*=\cup_{k\in\Z}\D_k$ in $X$ with parameters $c_0$, $C_0$, $\delta$.

We first assume that we are in the easier case (II). We define the maximal operator $\M$ with respect to the cubes $Q\in\D^*$ by $\M f:=\sup_{Q\in\D^*}\langle f\rangle_{1,Q}\chi_Q$.
\begin{lemma}[Calder\'on-Zygmund Lemma in the case (II)]\label{lem:caldzyg2}
Let $f\in L^1$, $\lambda>0$, and let $\Omega:=\{\M f>\lambda\}$. If $\Omega\neq X$, then we can find a pairwise disjoint collection of cubes $\mathscr{P}\subseteq\D^*$ and a constant $c>0$, depending only on the parameters of the dyadic system, the doubling dimension $\nu$, and the quasimetric constant $A$, so that
\[
\Omega=\bigcup_{P\in\mathscr{P}}P,
\]
and for each $P\in\mathscr{P}$
\[
\lambda<\langle f\rangle_{1,P}\leq c\lambda.
\]
\end{lemma}
\begin{proof}
Fix $k_0\in\Z$ small enough so that $c_0\delta^{k_0}>\diam X$. Then for any $x\in X$ we have $B(x;c_0\delta^{k_0})=X$. Hence, it follows from property (iii) of dyadic systems that $\D_{k_0}=\{X\}$.

Note that $\Omega\neq X$ implies that $\langle f\rangle_{1,X}\leq\lambda$. Let $x\in\Omega$. Then the set
\[
K_x:=\{k>k_0\mid\text{there is a $Q\in\D_k$,\, $x\in Q$,\, $\langle f\rangle_{1,Q}>\lambda$}\}
\]
is non-empty. Thus, by well-orderedness there is a minimal $k_x\in K_x$, and thus a cube $P_x\in\D_{k_x}$ that contains $x$ so that $\langle f\rangle_{1,P_x}>\lambda$. By minimality of $k_x$, it follows that $\langle f\rangle_{1,p(P_x)}\leq\lambda$, where $p(P_x)\in\D_{k_x-1}$ denotes the parent of $P_x$. By \eqref{eq:doublingpropnutwo} and property (iii) of dyadic systems this implies that
\[
\lambda<\langle f\rangle_{1,P_x}\leq c\langle f\rangle_{1,p(P_x)}\leq c\lambda,
\]
with $c=C(2AC_0/(c_0\delta))^\nu$.

It remains to show that the hereby obtained collection $\mathcal{P}=(P_x)_{x\in X}$ is pairwise disjoint. Indeed, assume that $P_1,P_2\in\mathcal{P}$ so that $P_1\cap P_2\neq\emptyset$. We have either $P_1\subseteq P_2$ or $P_2\subseteq P_1$ by property (ii) of dyadic systems. Without loss of generality we assume the first. Pick $x\in X$ so that $P_1=P_x$. Since $x\in P_2$ and $\langle f\rangle_{1,P_2}>\lambda$, minimality of $k_x$ implies that $P_2\in\D_l$ for some $l\geq k_x$. Again by property (ii) of dyadic systems, this implies that $P_2\subseteq P_1$, proving that $P_1=P_2$. The assertion follows.
\end{proof}

Next, we consider the case (I). We define the maximal operator $\M^{\mathscr{B}}$ with respect to the balls $B\subseteq X$ by $\M^{\mathscr{B}} f:=\sup_{B}\langle f\rangle_{1,B}\chi_B$.
\begin{lemma}[Calder\'on-Zygmund Lemma in the case (I)]\label{lem:caldzyg1}
Let $f\in L^1$, $\lambda>0$, and let $\Omega:=\{\M^{\mathscr{B}} f>\lambda\}$. If $\Omega\neq X$, then we can find a pairwise disjoint collection of cubes $\mathscr{P}\subseteq\D^*$ and a constant $c>0$, depending only on the parameters of the dyadic system, the doubling dimension $\nu$, the quasimetric constant $A$, and $\gamma$, so that
\[
\Omega=\bigcup_{P\in\mathscr{P}}P,
\]
and for each $P\in\mathscr{P}$
\[
\langle f\rangle_{1,P}\leq c\lambda.
\]
\end{lemma}
For the proof we use a version of the Whitney Decomposition Theorem. Note that the diameter assumption \eqref{eq:gamma} together with property (iii) of dyadic systems implies that for any $Q\in\D_k$ we have
\begin{equation}\label{eq:gammaq}
\gamma c_0\delta^k\leq\diam Q\leq 2AC_0\delta^k.
\end{equation}
\begin{theorem}[Whitney Decomposition Theorem for Dyadic Cubes]\label{thm:whitney}
Let $\Omega\subsetneq X$ be open. Then there exists a pairwise disjoint collection of cubes $\mathscr{P}\subseteq\D^*$ such that
\[
\Omega=\bigcup_{P\in\mathscr{P}}P
\]
and for each $P\in\mathscr{P}$,
\[
\diam P\leq d(P,\Omega^c)\leq\frac{4A^2C_0}{\gamma c_0\delta}\diam P.
\]
\end{theorem}
\begin{proof}
We define
\[
\mathscr{E}:=\{Q\in\D^*\mid Q\subseteq\Omega,\,\diam Q\leq d(Q,\Omega^c)\}.
\]
Moreover we set
\[
\mathscr{P}:=\{Q\in\mathscr{E}\mid\text{there is a $k\in\Z$ so that }Q\in\D_k,\,p(Q)\notin\mathscr{E}\},
\]
where $p(Q)\in\D_{k-1}$ denotes the parent of $Q\in\D_k$. We will show that
\[
\bigcup_{P\in\mathscr{P}}P=\Omega.
\]
Indeed, any $P\in\mathscr{P}$ is contained in $\Omega$. Conversely, if $x\in\Omega$, Let $(Q^k_x)_{k\in\Z}$ be the sequence of cubes in $\D^*$ with $x\in Q^k_x$ and $Q^k_x\in\D_k$ for all $k\in\Z$. Since $\Omega$ is open, there is a ball $B=B(x;r)$ contained in $\Omega$. Picking $k_0$ large enough so that $2AC_0\delta^{k_0}<r$, we find that
\[
Q^k_x\subseteq B(x;r)\subseteq\Omega
\]
for all $k\geq k_0$ by \eqref{eq:gammaq}. Moreover, since $d(Q^k_x,\Omega^c)\geq A^{-1}(d(x,\Omega^c)-2A^2C_0\delta^k)\uparrow A^{-1}d(x,\Omega^c)$ as $k\to\infty$, while $\diam(Q^k_x)\leq 2AC_0\delta^k\downarrow 0$ as $k\to\infty$, we can find a $k_1\in\Z$ so that $\diam(Q^k_x)\leq d(Q^k_x,\Omega^c)$ whenever $k\geq k_1$. Hence, for all $k\geq\max(k_0,k_1)$ we have $Q^k_x\in\mathscr{E}$. Thus, the set
\[
K_x:=\{k\in\Z\mid Q^k_x\in\mathscr{E}\}
\]
is non-empty. We also claim that $K_x$ is bounded from below. Indeed, if we choose $k_2\in\Z$ small enough so that $\gamma c_0\delta^{k_2}>d(x,\Omega^c)$, then
\[
d(Q^k_x,\Omega^c)\leq d(x,\Omega^c)<\diam(Q^k_x)
\]
for all $k\leq k_2$ by \eqref{eq:gammaq}, and hence $Q^k_x\notin\mathscr{E}$ for $k\leq k_2$, proving the claim.

We set $k_x:=\min K_x\in\Z$. Then $Q^{k_x}_x\in\mathscr{E}$ while $p(Q^{k_x}_x)=Q^{k_x-1}_x\notin\mathscr{E}$. Hence, $Q^{k_x}_x\in\mathscr{P}$, proving that $x\in\cup_{P\in\mathscr{P}}P$, as desired.

Next we will show that $\mathscr{P}$ is pairwise disjoint. Suppose for a contradiction that we have $P_1,P_2\in\mathscr{P}$ so that $P_1\cap P_2\neq\emptyset$ and $P_1\neq P_2$. Let $l_1,l_2\in\Z$ so that $P_1\in\D_{l_1}$, $P_2\in\D_{l_2}$ and $p(P_1),p(P_2)\notin\mathscr{E}$. Without loss of generality we assume that $l_1>l_2$ and thus $P_1\subseteq P_2$ by property (ii) of the dyadic systems. Then also $p(P_1)\subseteq P_2$. Since $p(P_1)\notin\mathscr{E}$, we must have that either $p(P_1)\nsubseteq\Omega$ or $d(p(P_1),\Omega^c)<d(p(P_1))$. The first case implies that $P_2\nsubseteq\Omega$, contradicting the fact that $P_2\in\mathscr{E}$. The second case implies that
\[
\diam(P_2)\geq \diam(p(P_1))> d(p(P_1),\Omega^c)\geq d(P_2,\Omega^c),
\]
again contradicting $P_2\in\mathscr{E}$. We conclude that $\mathscr{P}$ is pairwise disjoint, as desired.

It remains to show that $d(P,\Omega^c)<4A^2C_0/(\gamma c_0\delta)\diam P$ for all $P\in\mathscr{P}$. Let $P\in\mathscr{P}$, $P\in\D_k$ so that $p(P)\notin\mathscr{E}$. Then either $p(P)\nsubseteq\Omega$ or $d(p(P),\Omega^c)<\diam(p(P))$. In the first case we have $d(p(Q),\Omega^c)=0$, so in both cases we have
\[
d(p(P),\Omega^c)<\diam(p(P))\leq 2AC_0\delta^{k-1}=\frac{2AC_0}{\gamma c_0\delta}\gamma c_0\delta^k\leq\frac{2AC_0}{\gamma c_0\delta}\diam P.
\]
by \eqref{eq:gammaq}. Hence,
\[
d(P,\Omega^c)\leq A(d(p(P),\Omega^c)+\diam(p(P)))<\frac{4A^2C_0}{\gamma c_0\delta}\diam P,
\]
as desired.
\end{proof}
\begin{proof}[Proof of Lemma \ref{lem:caldzyg1}]
We apply the Whitney Decomposition Theorem to write $\Omega=\cup_{P\in\mathscr{P}}P$.

If $P\in\mathscr{P}$, $P\in\D_k$ with center $z_P$, we have
\begin{align*}
2d(z_P,\Omega^c)&\leq 2A\diam P+2Ad(P,\Omega^c)\leq\left(2A+\frac{8A^3C_0}{\gamma c_0\delta}\right)\diam P\\
&\leq\left(4A+\frac{16A^3C_0}{\gamma c_0\delta}\right)C_0\delta^k=:\tau C_0\delta^k
\end{align*}
so that
\[
\emptyset\neq B(z_P;2d(z_p,\Omega^c))\cap\Omega^c\subseteq B(z_P;\tau C_0\delta^k)\cap\Omega^c.
\]
Since
\[
\left|B(z_P;\tau C_0\delta^k)\right|\leq C\left(\frac{\tau C_0}{c_0}\right)^\nu|B(z_P;c_0\delta^k)|\lesssim|P|
\]
by \eqref{eq:doublingpropnu}, we may pick a point $x\in B(z_p;\tau C_0\delta^k)\cap\Omega^c$ to conclude that
\[
\langle f\rangle_{1,P}\lesssim\langle f\rangle_{1,B(z_p;\tau C_0\delta^k)}\leq\M^{\mathscr{B}}f(x)\leq\lambda.
\]
The assertion follows.
\end{proof}

\begin{proof}[Proof of Theorem \ref{thm:thingswork}]
In both cases (I) and (II), the proof of Theorem \ref{thm:main1} holds mutatis mutandis. Moreover, in the case (I), the same is true for Theorem \ref{thm:main2}, where one uses Lemma \ref{lem:caldzyg1}, and for Theorem \ref{thm:maindual}, where one uses Theorem \ref{thm:whitney}.

For Theorem \ref{thm:main2} in the case (II), one replaces the set $\Omega$ in the proof by the set $\Omega=\{\M(|f|^{p_0})> 2[w]_{A_1} w(E)^{-1}\}$. We claim that $\Omega\neq X$. Indeed, since $X$ is bounded, we have $w(X)<\infty$. Thus, by \eqref{eq:classfefstein}, we have
\[
w(\Omega)\leq\frac{w(E)}{2[w]_{A_1}}\int\!|f|^{p_0}\M w\,\mathrm{d}\mu\leq\frac{w(E)}{2}\leq\frac{w(X)}{2}<w(X),
\]
proving the claim. Thus we may apply Lemma \ref{lem:caldzyg2} to decompose $\Omega$, and the remainder of the proof runs analogously.
\end{proof}

\section{Optimality of weighted strong type estimates}
In this section we are going to show that the weighted strong type estimates in \eqref{eq:atwo} and \eqref{eq:bestimate} are optimal, given a certain asymptotic behaviour  of the unweighted $L^p$ operator norm of $T$. Such asymptotic behaviour is directly linked to lower bounds on the (generalized) kernel of the operator, see Example \ref{ex:twocopies}. We improve upon the result in \cite{frey16}, where it was shown that the estimate \eqref{eq:atwo} is optimal for sparse forms. Indeed, here we are directly using properties of the operator $T$ itself rather than only its sparse bounds.

Our method is an adaptation of the results of Fefferman, Pipher \cite{fefpip97} and Luque, P\'erez and Rela \cite{luque15}. We deduce sharpness of weighted bounds from the asymptotic behaviour of the unweighted $L^p$ norm of $T$ as $p$ tends to $p_0$ and $q_0$, respectively. The proof exploits the known sharp behaviour of the Hardy-Littlewood maximal function via the iteration algorithm of Rubio de Francia.

We will work in a doubling metric measure space $(X,d,\mu)$ satisfying the assumptions from the Section \ref{sec:ext}. As a matter of fact, the only property we need is a precise control of the $L^p$ norm of the maximal operator. More precisely, we let $\D:=\cup_{\alpha=1}^K\D^\alpha$ be the union of the dyadic grids in $X$ obtained from Theorem \ref{thm:dyadicgrids}. Then we define
\[
\M_qf:=\max_{1\leq\alpha\leq K}\sup_{Q\in\D^\alpha}\langle f\rangle_{q,Q}\chi_Q=\sup_{Q\in\D}\langle f\rangle_{q,Q}\chi_Q
\]
for $1\leq q<\infty$, where we set $\M:=\M_1$. Using the shorthand notation $\|\M_q\|_p=\|\M_q\|_{L^p\to L^p}$ for $p>q$, we will use
\begin{equation}\label{eq:maxfinsec}
\|\M\|_p\leq Kp',
\end{equation}
which follows as in \eqref{eq:fefstein} with $w=1$.

Let us first define the critical exponents that determine the asymptotic behaviour of the unweighted $L^p$ operator norm of $T$.

\begin{definition} \label{def:exponents}
Let $1 \leq p_0 < q_0 \leq \infty$.
Let  $T$ be a bounded operator on $L^p$ for all  $p_0<p<q_0$.  We define
\[
	\alpha_T(p_0) := \sup \{ \alpha \geq 0 \,\mid\, \forall \eps > 0, \, \limsup_{p \to p_0} \, (p-p_0)^{\alpha-\epsilon} \|T\|_{L^p\to L^p} = \infty\}.
\]
For $q_0<\infty$ we define
\[
	\gamma_T(q_0) := \sup \{ \gamma \geq 0 \,\mid\, \forall \eps >0, \, \limsup_{p \to q_0} \,(q_0-p)^{\gamma-\eps}\|T\|_{L^p\to L^p}  = \infty\},
\]
and for $q_0=\infty$
\[
	\gamma_T(\infty) := \sup \{ \gamma \geq 0 \,\mid\, \forall \eps >0, \, \limsup_{p \to \infty} \, \frac{\|T\|_{L^p\to L^p} }{ p^{\gamma-\eps}} = \infty\}.
\]
\end{definition}

For $p_0<s<q_0$ we define
\[
\phi(s):=\left(\frac{q_0}{s}\right)'\left(\frac{s}{p_0}-1\right)+1.
\]
Then it follows from Proposition \ref{prop:weightprop}(ii) that for a weight $w$ we have $w\in A_{s/p_0}\cap\RH_{(q_0/s)'}$ if and only if $w^{(q_0/s)'}\in A_{\phi(s)}$.

 We establish the following connection between the weighted strong type estimates for $T$ and the asymptotic behaviour of the unweighted $L^p$ operator norm  at the endpoints $p=p_0$ and $p=q_0$.

\begin{theorem} \label{thm:ap-opt}
Let $T$ be a bounded operator on $L^p$ for all $p_0<p<q_0$. Suppose that for some $p_0<s<q_0$ and for all $w\in A_{s/p_0}\cap\RH_{(q_0/s)'}$,
\begin{equation} \label{eq:eq-ap}
	\|T\|_{L^s(w) \to L^s(w)} \leq c [w^{(q_0/s)'}]_{A_{\phi(s)}}^{\beta/(q_0/s)'}.
\end{equation}
Then
\[
\beta\geq\max\left(\frac{p_0}{s-p_0}\alpha_T(p_0),\left(\frac{q_0}{s}\right)'\gamma_T(q_0) \right).
\]
\end{theorem}

We also establish a version involving the $A_1$ characteristics. Its proof follows the same lines as the one for Theorem \ref{thm:ap-opt} and will therefore be omitted.

\begin{theorem} \label{thm:a1-opt}
Let $T$ be a bounded operator on $L^p$ for all $p_0<p<q_0$. Suppose that for some $p_0<s<q_0$ and for all $w\in A_1\cap\RH_{(q_0/s)'}$,
\begin{equation} \label{eq:eq-a1}
\|T\|_{L^s(w) \to L^s(w)} \leq c [w^{(q_0/s)'}]_{A_1}^{\beta/(q_0/s)'}.
\end{equation}
Then
\[
\beta\geq\left(\frac{q_0}{s}\right)'\gamma_T(q_0).
\]
\end{theorem}

\begin{proof}[Proof of Theorem \ref{thm:ap-opt}]
We adapt the proof of  \cite{luque15}, which is based on the iteration algorithm of Rubio de Francia. Let $p_0<p<s$, and define the operator $\calR$ by
\[
		\calR h = \sum_{k=0}^\infty \frac{1}{2^k} \frac{\M_{p_0}^k h}{\|\M_{p_0}\|_p^k}
\]
for $h \in L^p$ with $h\geq 0$.
We claim that then
\begin{enumerate}
\item[(A)] $h\leq \calR h$,
\item[(B)] $\|\calR h\|_p \leq 2 \|h\|_p$,
\item[(C)] $[(\calR h)^{p_0}]_{A_1} \leq 2^{p_0} \|\M\|_{p/p_0}$.
\end{enumerate}
Properties (A) and (B) are immediate. Note that (B) uses the assumption $p_0<p$. Property (C) can be seen as follows:
By definition of $\calR$,  we have
\[
\M_{p_0}(\calR h) \leq 2 \|\M_{p_0}\|_{p} \calR h.
\]
Thus, using that $\|\M_{p_0}\|_{p} = \|\M\|_{p/p_0}^{1/p_0}$, we obtain for $v=\calR h$
\[
		\M(v^{p_0}) = (\M_{p_0}v)^{p_0}
		\leq (2\|\M_{p_0}\|_{p} v)^{p_0}
		=2^{p_0} \|\M\|_{p/p_0} v^{p_0},
\]
which yields (C).

Let us now estimate $\|T\|_{L^p\to L^p}$, given \eqref{eq:eq-ap}.
Let $f \in L^p$. Then by H\"older's inequality,
\begin{align*}
	\|Tf\|_p
	&= \left(\int |Tf|^p (\calR |f|)^{-(s-p)\frac{p}{s}} (\calR |f|)^{(s-p)\frac{p}{s}} \,\mathrm{d}\mu\right)^{1/p} \\
	& \leq \left(\int |Tf|^s (\calR |f|)^{-(s-p)} \,\mathrm{d}\mu\right)^{1/s}
	\left(\int (\calR |f|)^p \,\mathrm{d}\mu\right)^{\frac{s-p}{ps}}.
\end{align*}
We abbreviate $w:=(\calR |f|)^{-(s-p)}$. Applying assumption \eqref{eq:eq-ap} and (B) in the first step and (A) in the second step yields
\begin{align*}
	\|Tf\|_p
	& \leq c [w^{(q_0/s)'}]_{A_{\phi(s)}}^{\beta/(q_0/s)'} \left(\int |f|^s w \,\mathrm{d}\mu\right)^{1/s} \|f\|_p^{\frac{s-p}{s}} \\
	& \lesssim [w^{(q_0/s)'}]_{A_{\phi(s)}}^{\beta/(q_0/s)'} \left(\int |f|^p \,\mathrm{d}\mu\right)^{1/s} \|f\|_p^{1-\frac{p}{s}} \\
	& = c[w^{(q_0/s)'}]_{A_{\phi(s)}}^{\beta/(q_0/s)'}\|f\|_p.\\
	& = c[w^{(q_0/s)'(1-{\phi(s)}')}]_{A_{{\phi(s)}'}}^{({\phi(s)}-1)\beta/(q_0/s)'} \|f\|_p,
\end{align*}
since $[w]_{A_q} = [w^{1-q'}]_{A_{q'}}^{q-1}$.
Using the equality
\[
		\left(\frac{q_0}{s}\right)' \frac{s-p}{{\phi(s)} -1 } = p_0 \frac{s-p}{s-p_0}
\]
and Jensen's inequality with exponent $\frac{s-p}{s-p_0} <1$, we can write
\begin{align*}
	[w^{(q_0/s)'(1-{\phi(s)}')}]_{A_{{\phi(s)}'}}
	 = [(\calR |f|)^{(q_0/s)'\frac{s-p}{{\phi(s)}-1}}]_{A_{{\phi(s)}'}}
	\leq [(\calR |f|)^{p_0}]_{A_{{\phi(s)}'}}^{\frac{s-p}{s-p_0}}.	
\end{align*}
We thus obtain
\begin{align*}
	\|Tf\|_p
	& \lesssim
	 [(\calR |f|)^{p_0}]_{A_{{\phi(s)}'}}^{\beta \frac{s-p}{p_0} } \|f\|_p
	\leq [(\calR |f|)^{p_0}]_{A_1}^{\beta \frac{s-p}{p_0} }\|f\|_p.
\end{align*}
From property (C) and \eqref{eq:maxfinsec} we can then deduce
\begin{align*}
	\|T\|_{L^p\to L^p}
	\lesssim \|\M\|_{p/p_0}^{\beta \frac{s-p}{p_0}}
	\lesssim \left(\frac{p}{p-p_0}\right)^{\beta \frac{s-p}{p_0}}.
\end{align*}
This shows
\[
\limsup_{p\to p_0}(p-p_0)^{\beta \frac{s-p}{p_0}}\|T\|_{L^p\to L^p}\lesssim\limsup_{p\to p_0}p^{\beta \frac{s-p}{p_0}}<\infty
\]
which by definition of $\alpha_T(p_0)$ implies that $\beta \geq \frac{p_0}{s-p_0} \alpha_T(p_0)$.

Now for the behaviour for $p \to q_0$ we follow the argument of \cite{fefpip97}. We assume $q_0<\infty$. The case $q_0=\infty$ has been treated in \cite{luque15} already.
We abbreviate $q:=(q_0/s)'$.
Let $p_0<s<p<q_0$. We again use the iteration algorithm of Rubio de Francia, but slightly change the definition of the operator $\calR$. This time, we define $\calR$ by
\[
\calR h = \sum_{k=0}^\infty \frac{1}{2^k} \frac{\M_{q} ^k h}{\|\M_{q}\|_{(p/s)'}^k}
\]
for $h \in L^{(p/s)'}$ with $h\geq 0$.	
Then, just as before, we have
\begin{enumerate}
\item[(A)] $ h \leq \calR h$,
\item[(B)] $\|\calR h\|_{(p/s)'} \leq 2 \|h\|_{(p/s)'}$,
\item[(C)] $[(\calR h)^q]_{A_1} \leq 2^q \|M\|_{(p/s)'/q}$.
\end{enumerate}

Let $f \in L^p$. There exists $h \in L^{(p/s)'}$ with $\|h\|_{(p/s)'} =1$ and $h \geq 0$ so that by (A)
\begin{equation}\label{eq:tfsecest}
	\|Tf\|_p^s=\||Tf|^s\|_{p/s}\leq\int|Tf|^s h\,\mathrm{d}\mu
	\leq \int |Tf|^s \calR h \,\mathrm{d}\mu.
\end{equation}
It follows from the assumption \eqref{eq:eq-ap} and (B) that
\begin{align*}
\int |Tf|^s \calR h \,\mathrm{d}\mu
	& \leq c [(\calR h)^{(q_0/s)'}]_{A_{\phi(s)}}^{s\beta/(q_0/s)'} \int |f|^s \calR h \,\mathrm{d}\mu \\
	& \lesssim [(\calR h)^{(q_0/s)'}]_{A_1}^{s\beta/(q_0/s)'}  \|f\|_p^s.
\end{align*}
Hence, by \eqref{eq:tfsecest} and (C), we have
\[
		\|T\|_{L^p \to L^p} \lesssim [(\calR h)^{(q_0/s)'}]_{A_1}^{\beta/(q_0/s)'}
		\lesssim \|\M\|_{(p/s)'/(q_0/s)'}^{\beta/(q_0/s)'}.
\]
Using \eqref{eq:maxfinsec}, we find
\begin{align*}
	\|T\|_{L^p\to L^p} \lesssim \left( \frac{p(q_0-s)}{s(q_0-p)}\right)^{\beta/(q_0/s)'}
\end{align*}
so that
\[
\limsup_{p\to q_0}(q_0-p)^{\beta/(q_0/s)'}\|T\|_{L^p\to L^p}\lesssim\limsup_{p\to q_0}\left(\frac{p(q_0-s)}{s}\right)^{\beta/(q_0/s)'}<\infty.
\]
By definition of $\gamma_T(q_0)$, this yields $\beta\geq(q_0/s)' \gamma_T(q_0)$, proving the assertion.
\end{proof}

For the application of these results  to sparsely dominated operators,  we make the following observation.
\begin{proposition}
Let $1\leq p_0<q_0\leq\infty$ and let $T\in S(p_0,q_0)$. Then
\[
\alpha_T(p_0)\leq\frac{1}{p_0},\quad\gamma_T(q_0)\leq\frac{1}{q_0'}.
\]
\end{proposition}
\begin{proof}
This is an immediate consequence of the fact that
\[
\|T\|_{L^p\to L^p}\lesssim\left[\left(\frac{p'}{q_0'}\right)'\right]^{\frac{1}{q_0'}}\left[\left(\frac{p}{p_0}\right)'\right]^{\frac{1}{p_0}},
\]
which follows from Remark \ref{rem:unweightedopnorm}.
\end{proof}

From the above, we can deduce optimality of the weighted estimates as stated in Theorem \ref{thm:opti-end}.

\begin{proof}[Proof of Theorem \ref{thm:opti-end}]
Let $\beta$ denote the best constant in the estimate
\[
\|T\|_{L^p(w) \to L^p(w)}\lesssim[w^{(q_0/p)'}]_{A_{\phi(p)}}^{\beta/(q_0/p)'}.
\]
Then it follows from the result \eqref{eq:atwo} from \cite{frey16} that
\[
\beta\leq\max\left(\frac{1}{p-p_0},\frac{q_0-1}{q_0-p}\right)
\]
Conversely, it follows from Theorem \ref{thm:ap-opt} that
\[
\beta\geq\max\left(\frac{p_0}{p-p_0}\alpha_T(p_0),\left(\frac{q_0}{p}\right)'\gamma_T(q_0) \right)=\max\left(\frac{1}{p-p_0},\frac{q_0-1}{q_0-p}\right)
\]
proving the first result. Using Theorem \ref{thm:a1-opt}, the second result follows analogously.
\end{proof}

Let us give an example of an operator $T$ for which the exponent $\gamma_T(q_0)$ is known.

\begin{example} \label{ex:twocopies}
Let $M$ be a complete $C^\infty$ Riemannian manifold $M$ of dimension $n \geq 3$. Assume that $M$ is the union of a compact part and a finite number of Euclidean ends, e.g. two copies of $\R^n$ glued smoothly along their unit circles. Then it was shown in \cite{CCH06} that in the case that the number of ends is at least two, the corresponding Riesz transform $T$ is bounded from $L^p(M)$ to $L^p(M;T^\ast M)$ if and only if $1<p<n$. More precisely, it was shown in \cite[Lemma 5.1]{CCH06} that the kernel of $T$ decays only to order $n-1$. A straightforward calculation, analogous to the classical results (see e.g. \cite[p.42]{stein93}), shows that this implies $\gamma_T(q_0)=\gamma_T(n)=\frac{n-1}{n}$.
\end{example}

\section*{Acknowledgement}

The first author would like to thank Carlos P\'erez for pointing out the results of \cite{luque15}. The second author is grateful to Fr\'ed\'eric Bernicot and Jos\'e Manuel Conde-Alonso for pointing out weak type results for sparse operators.

\bibliographystyle{alpha}

\begin{thebibliography}{GCRdF85}

\bibitem[AM07]{auscher072}
P.~Auscher and J.M. Martell.
\newblock Weighted norm inequalities, off-diagonal estimates and elliptic
  operators. {I}. {G}eneral operator theory and weights.
\newblock {\em Adv. Math.}, 212(1):225--276, 2007.

\bibitem[AM08]{auscher08}
P.~Auscher and J.M. Martell.
\newblock Weighted norm inequalities, off-diagonal estimates and elliptic
  operators. {IV}. {R}iesz transforms on manifolds and weights.
\newblock {\em Math. Z.}, 260(3):527--539, 2008.

\bibitem[Aus07]{auscher07}
P.~Auscher.
\newblock On necessary and sufficient conditions for {$L^p$}-estimates of
  {R}iesz transforms associated to elliptic operators on {$\Bbb R^n$} and
  related estimates.
\newblock {\em Mem. Amer. Math. Soc.}, 186(871):xviii+75, 2007.

\bibitem[BB11]{bjorn11}
A.~Bj\"orn and J.~Bj\"orn.
\newblock {\em Nonlinear potential theory on metric spaces}, volume~17 of {\em
  EMS Tracts in Mathematics}.
\newblock European Mathematical Society (EMS), Z\"urich, 2011.

\bibitem[BB17]{benea17}
C.~{Benea} and F.~{Bernicot}.
\newblock {Conservation de certaines propri\'et\'es \`a travers un contr\^ole
  \'epars d'un op\'erateur et applications au projecteur de Leray-Hopf}.
\newblock {\em ArXiv:1703.00228}, March 2017.

\bibitem[BBL17]{benea171}
C.~Benea, F.~Bernicot, and T.~Luque.
\newblock Sparse bilinear forms for {B}ochner {R}iesz multipliers and
  applications.
\newblock {\em Trans. London Math. Soc.}, 4(1):110--128, 2017.

\bibitem[BFP16]{frey16}
F.~Bernicot, D.~Frey, and S.~Petermichl.
\newblock Sharp weighted norm estimates beyond {C}alder\'on-{Z}ygmund theory.
\newblock {\em Anal. PDE}, 9(5):1079--1113, 2016.

\bibitem[BK03]{kunstmann03}
S.~Blunck and P.C. Kunstmann.
\newblock Calder\'on-{Z}ygmund theory for non-integral operators and the
  {$H^\infty$} functional calculus.
\newblock {\em Rev. Mat. Iberoamericana}, 19(3):919--942, 2003.

\bibitem[CCDO16]{condealonso16}
J.M. {Conde-Alonso}, A.~{Culiuc}, F.~{Di Plinio}, and Y.~{Ou}.
\newblock {A sparse domination principle for rough singular integrals}.
\newblock {\em ArXiv:1612.09201}, December 2016.

\bibitem[CCH06]{CCH06}
G.~Carron, T.~Coulhon, and A.~Hassell.
\newblock Riesz transform and {$L^p$}-cohomology for manifolds with {E}uclidean
  ends.
\newblock {\em Duke Math. J.}, 133(1):59--93, 2006.

\bibitem[CDO16]{diplinio16}
A.~{Culiuc}, F.~{Di Plinio}, and Y.~{Ou}.
\newblock {Domination of multilinear singular integrals by positive sparse
  forms}.
\newblock {\em ArXiv:1603.05317}, March 2016.

\bibitem[CR80]{coifman80}
R.R. Coifman and R.~Rochberg.
\newblock Another characterization of {BMO}.
\newblock {\em Proc. Amer. Math. Soc.}, 79(2):249--254, 1980.

\bibitem[CUMP12]{cruz12}
D.~Cruz-Uribe, J.M. Martell, and C.~P\'erez.
\newblock Sharp weighted estimates for classical operators.
\newblock {\em Adv. Math.}, 229(1):408--441, 2012.

\bibitem[CW71]{CW71}
R.~R. Coifman and G.~Weiss.
\newblock {\em Analyse harmonique non-commutative sur certains espaces
  homog\`enes}.
\newblock Lecture Notes in Mathematics, Vol. 242. Springer-Verlag, Berlin-New
  York, 1971.
\newblock \'Etude de certaines int\'egrales singuli\`eres.

\bibitem[DHL17]{diplinio17}
F.~{Di Plinio}, T.P. {Hyt{\"o}nen}, and K.~{Li}.
\newblock {Sparse bounds for maximal rough singular integrals via the Fourier
  transform}.
\newblock {\em ArXiv:1706.09064}, June 2017.

\bibitem[FP97]{fefpip97}
R.~Fefferman and J.~Pipher.
\newblock Multiparameter operators and sharp weighted inequalities.
\newblock {\em Amer. J. Math.}, 119(2):337--369, 1997.

\bibitem[FS71]{fefstein71}
C.~Fefferman and E.M. Stein.
\newblock Some maximal inequalities.
\newblock {\em Amer. J. Math.}, 93:107--115, 1971.

\bibitem[GCRdF85]{garcia85}
J.~Garc{\'\i}a-Cuerva and J.L. Rubio~de Francia.
\newblock {\em Weighted norm inequalities and related topics}, volume 116 of
  {\em North-Holland Mathematics Studies}.
\newblock North-Holland Publishing Co., Amsterdam, 1985.
\newblock Notas de Matem\'atica [Mathematical Notes], 104.

\bibitem[Gra08]{grafakosclassic}
L.~Grafakos.
\newblock {\em Classical {F}ourier analysis}, volume 249 of {\em Graduate Texts
  in Mathematics}.
\newblock Springer, New York, second edition, 2008.

\bibitem[Gra09]{grafakosmodern}
L.~Grafakos.
\newblock {\em Modern {F}ourier analysis}, volume 250 of {\em Graduate Texts in
  Mathematics}.
\newblock Springer, New York, second edition, 2009.

\bibitem[Gri91]{grig91}
A.A. Grigor'yan.
\newblock The heat equation on noncompact {R}iemannian manifolds.
\newblock {\em Mat. Sb.}, 182(1):55--87, 1991.

\bibitem[HK00]{hajlasz00}
P.~Haj{\l}asz and P.~Koskela.
\newblock Sobolev met {P}oincar\'e.
\newblock {\em Mem. Amer. Math. Soc.}, 145(688):x+101, 2000.

\bibitem[HK12]{hytonen12}
T.P. Hyt{\"o}nen and A.~Kairema.
\newblock Systems of dyadic cubes in a doubling metric space.
\newblock {\em Colloq. Math.}, 126(1):1--33, 2012.

\bibitem[HM03]{HM03}
S.~Hofmann and J.M. Martell.
\newblock {$L^p$} bounds for {R}iesz transforms and square roots associated to
  second order elliptic operators.
\newblock {\em Publ. Mat.}, 47(2):497--515, 2003.

\bibitem[HP13]{hytonen13}
T.P. Hyt\"onen and C.~P\'erez.
\newblock Sharp weighted bounds involving {$A_\infty$}.
\newblock {\em Anal. PDE}, 6(4):777--818, 2013.

\bibitem[HPR12]{hytonen122}
T.P. Hyt\"onen, C.~P\'erez, and E.~Rela.
\newblock Sharp reverse {H}\"older property for {$A_\infty$} weights on spaces
  of homogeneous type.
\newblock {\em J. Funct. Anal.}, 263(12):3883--3899, 2012.

\bibitem[Hyt12]{hytonen123}
T.P. Hyt\"onen.
\newblock The sharp weighted bound for general {C}alder\'on-{Z}ygmund
  operators.
\newblock {\em Ann. of Math. (2)}, 175(3):1473--1506, 2012.

\bibitem[JN91]{johnson91}
R.~Johnson and C.J. Neugebauer.
\newblock Change of variable results for {$A_p$}- and reverse {H}\"older {${\rm
  RH}_r$}-classes.
\newblock {\em Trans. Amer. Math. Soc.}, 328(2):639--666, 1991.

\bibitem[Lac17]{lacey17}
Michael~T. Lacey.
\newblock An elementary proof of the {$A_2$} bound.
\newblock {\em Israel J. Math.}, 217(1):181--195, 2017.

\bibitem[Ler10]{lerner10}
A.K. Lerner.
\newblock A pointwise estimate for the local sharp maximal function with
  applications to singular integrals.
\newblock {\em Bull. Lond. Math. Soc.}, 42(5):843--856, 2010.

\bibitem[Li17]{li17}
K.~Li.
\newblock Two weight inequalities for bilinear forms.
\newblock {\em Collect. Math.}, 68(1):129--144, 2017.

\bibitem[LN15]{lernaz15}
A.K. {Lerner} and F.~{Nazarov}.
\newblock {Intuitive dyadic calculus: the basics}.
\newblock {\em ArXiv:1508.05639}, August 2015.

\bibitem[LOP08]{lerner08}
A.K. Lerner, S.~Ombrosi, and C.~P\'erez.
\newblock Sharp {$A_1$} bounds for {C}alder\'on-{Z}ygmund operators and the
  relationship with a problem of {M}uckenhoupt and {W}heeden.
\newblock {\em Int. Math. Res. Not. IMRN}, (6):Art. ID rnm161, 11, 2008.

\bibitem[LOP09]{lerner09}
A.K. Lerner, S.~Ombrosi, and C.~P\'erez.
\newblock Weak type estimates for singular integrals related to a dual problem
  of {M}uckenhoupt-{W}heeden.
\newblock {\em J. Fourier Anal. Appl.}, 15(3):394--403, 2009.

\bibitem[LPR15]{luque15}
T.~Luque, C.~P\'erez, and E.~Rela.
\newblock Optimal exponents in weighted estimates without examples.
\newblock {\em Math. Res. Lett.}, 22(1):183--201, 2015.

\bibitem[NRVV10]{nazarov10}
F.L. Nazarov, A.~Reznikov, V.~Vasyunin, and A.~Volberg.
\newblock Weak norm estimates of weighted singular operators and bellman
  functions.
\newblock {\em Preprint.
  \url{https://sashavolberg.files.wordpress.com/2010/11/a11_7loghilb11_21_2010.pdf}},
  2010.

\bibitem[P{\'e}r94]{perez94}
C.~P{\'e}rez.
\newblock Weighted norm inequalities for singular integral operators.
\newblock {\em J. London Math. Soc. (2)}, 49(2):296--308, 1994.

\bibitem[PV02]{petermichl02}
S.~Petermichl and A.~Volberg.
\newblock Heating of the {A}hlfors-{B}eurling operator: weakly quasiregular
  maps on the plane are quasiregular.
\newblock {\em Duke Math. J.}, 112(2):281--305, 2002.

\bibitem[Ste93]{stein93}
E.M. Stein.
\newblock {\em Harmonic analysis: real-variable methods, orthogonality, and
  oscillatory integrals}, volume~43 of {\em Princeton Mathematical Series}.
\newblock Princeton University Press, Princeton, NJ, 1993.
\newblock With the assistance of Timothy S. Murphy, Monographs in Harmonic
  Analysis, III.

\bibitem[Wil87]{wilson87}
J.M. Wilson.
\newblock Weighted inequalities for the dyadic square function without dyadic
  {$A_\infty$}.
\newblock {\em Duke Math. J.}, 55(1):19--50, 1987.

\bibitem[Wil89]{wilson89}
J.M. Wilson.
\newblock Weighted norm inequalities for the continuous square function.
\newblock {\em Trans. Amer. Math. Soc.}, 314(2):661--692, 1989.

\bibitem[Wil08]{wilson08}
J.M. Wilson.
\newblock {\em Weighted {L}ittlewood-{P}aley theory and exponential-square
  integrability}, volume 1924 of {\em Lecture Notes in Mathematics}.
\newblock Springer, Berlin, 2008.

\bibitem[WY13]{wang13}
F.-Y. Wang and L.~Yan.
\newblock Gradient estimate on convex domains and applications.
\newblock {\em Proc. Amer. Math. Soc.}, 141(3):1067--1081, 2013.

\end{thebibliography}

\end{document}